\pdfmapfile{=<name>.map}
\documentclass[12pt,oneside,reqno,fleqn]{amsart}
\usepackage[T1]{fontenc}
\usepackage[utf8]{inputenc}
\usepackage{amssymb}
\usepackage{bbm}
\usepackage{cases}
\usepackage{amsmath}
\usepackage{graphicx}
\usepackage{mathrsfs}
\usepackage{stmaryrd}
\usepackage{color}
\usepackage{soul}
\usepackage[dvipsnames]{xcolor}
\usepackage{amsfonts}
\usepackage{amsmath,amssymb,amsthm}
\usepackage{libertine}
\usepackage[T1]{fontenc}
\usepackage[libertine,varbb]{newtxmath}
\usepackage[colorlinks,linkcolor=red,citecolor=blue,urlcolor=red]{hyperref}
\usepackage[shortlabels] {enumitem}
\usepackage[nameinlink,capitalize]{cleveref}
\usepackage[spacing=true,kerning=true,tracking=true,letterspace=50]{microtype}
\hypersetup{pdftitle={Lagrangian uniqueness}, pdfauthor={LG}}
\numberwithin{equation}{section}
\newcommand{\be}{\begin{eqnarray}}
\newcommand{\ee}{\end{eqnarray}}
\newcommand{\ce}{\begin{eqnarray*}}
\newcommand{\de}{\end{eqnarray*}}
\newtheorem{theorem}{Theorem}[section]
\newtheorem{lemma}[theorem]{Lemma}
\newtheorem{proposition}[theorem]{Proposition}
\newtheorem{conjecture}[theorem]{Conjecture}
\newtheorem{corollary}[theorem]{Corollary}
\theoremstyle{remark}

\newtheorem{example}[theorem]{Example}
\newtheorem{remark}[theorem]{Remark}
\newtheorem{definition}[theorem]{Definition}

\crefname{eqn}{Equation}{Equations}
\crefname{assumption}{Assumption}{Assumptions}
\crefname{innercustomthm}{Condition}{Conditions}
\crefrangelabelformat{innercustomthm}{#3#1#4-#5#2#6}

\def\eps{\varepsilon}

\def\d{\delta}

\def\<{{\langle}}
\def\>{{\rangle}}
\def\({{\Big(}}
\def\){{\Big)}}

\def\bx{{\mathbf{x}}}

\def\Law{{\mathord{{\rm Law}}}}

\def\={&\!\!=\!\!&}
\def\bt{\begin{theorem}}
\def\et{\end{theorem}}
\def\bl{\begin{lemma}}
\def\el{\end{lemma}}
\def\br{\begin{remark}}
\def\er{\end{remark}}
\def\bd{\begin{definition}}
\def\ed{\end{definition}}
\def\bp{\begin{proposition}}
\def\ep{\end{proposition}}
\def\bc{\begin{corollary}}
\def\ec{\end{corollary}}
\def\bx{\begin{example}}
\def\ex{\end{example}}
\def\cA{{\mathcal A}}
\def\cB{{\mathcal B}}
\def\cC{{\mathcal C}}

\def\cE{{\mathcal E}}
\def\cF{{\mathcal F}}

\def\cM{{\mathcal M}}

\def\cP{{\mathcal P}}

\def\cR{{\mathcal R}}

\def\mE{{\mathbb E}}
\def\E{\mE}

\def\PP{{\mathbb P}}

\def\geq{\geqslant}
\def\leq{\leqslant}

\def\div{\mathord{{\rm div}}}

\newcommand{\dd}{\mathop{}\!\mathrm{d}}
\newcommand{\loc}{\mathrm{loc}}
\newcommand{\R}{{\mathbb R}}

\newcommand{\N}{{\mathbb{N}}}

\newcommand{\norm}[1]{{\left\vert\kern-0.25ex\left\vert\kern-0.25ex\left\vert #1
    \right\vert\kern-0.25ex\right\vert\kern-0.25ex\right\vert}}

\allowdisplaybreaks

\begin{document}
	\title{Almost-everywhere uniqueness of Lagrangian trajectories for $3$D Navier--Stokes revisited}
	\date{\today}
	\author{Lucio Galeati}
\address{Lucio Galeati, 
Dipartimento di Ingegneria e Scienze dell’Informazione e Matematica, Università degli Studi `
dell’Aquila, Italy\newline
\indent Email:  lucio.galeati@univaq.it
}

	\begin{abstract}
	We show that, for any Leray solution $u$ to the $3$D Navier--Stokes equations with $u_0\in L^2$, the associated deterministic and stochastic Lagrangian trajectories are unique for \textit{Lebesgue a.e.} initial condition. Additionally, if $u_0\in H^{1/2}$, then pathwise uniqueness is established for the stochastic Lagrangian trajectories starting from \textit{every} initial condition.
	The result sharpens and extends the original one by Robinson and Sadowski \cite{RobSad2009b} and is based on rather different techniques. A key role is played by a newly established asymmetric Lusin--Lipschitz property of Leray solutions $u$, in the framework of (random) Regular Lagrangian flows.\\[1ex]
	\\[1ex]
	\end{abstract}

\maketitle	

\section{Introduction}

In this paper we consider the Navier--Stokes equations on $\R^3$:
\begin{equation}\label{eq:NS_intro}\tag{NS}\begin{cases}
	\partial_t u + (u\cdot\nabla)u = \nu\Delta u + \nabla p\\
	\nabla\cdot u = 0, \quad u\vert_{t=0}=u_0.
\end{cases}\end{equation}
Their global solvability is a well-known oustanding open problem.
For sufficiently regular $u_0$, smooth solutions exist locally in time, but it is unknown whether they may develop singularities in finite time and whether they can be uniquely continued as weak solutions after blow-up. On the other hand, very weak distributional solutions are known to be non-unique by convex integration schemes \cite{BucVic2019}.
In between these two extremes, global existence of so called Leray solutions is known for any divergence-free $u_0\in L^2$; we refer to Definition \ref{defn:leray_solution} below for the exact notion\footnote{Please notice that Definition \ref{defn:leray_solution}, although classical, is stronger than other definitions of Leray solutions existing in the literature, see the discussion in Section \ref{subsec:main.results}.} and a discussion of their physical properties.
Uniqueness of Leray solutions is another major open problem; see however the recent breakthrough \cite{AlBrCo2022} for a non-uniqueness result in the presence of a well-chosen forcing term $f\in L^1([0,T];L^2)$.
We refer the reader to the monographs \cite{RRS2016,BedVic2022} for a nice exposition of relevant results and the state of the art for \eqref{eq:NS_intro}.

Under suitable regularity assumptions on $u$, Constantin and Iyer \cite{ConIye2008} established the representation formula
\begin{equation}\label{eq:constantin_iyer}
	u_t(x)=\Pi\, \E\Big[  (\nabla^T X^{-1}_t)(x) u_0(X^{-1}_t(x)) \Big]
\end{equation}
where $\Pi$ denotes the Leray projector, $\E$ is the expectation w.r.t. an underlying probability space and $X^{-1}$ is the inverse (or backward) flow associated to the SDE
\begin{equation}\label{eq:constantin_iyer2}
\dd X_t(x) = u_t(X_t(x))\dd t + \sqrt{2\nu} \dd W_t, \quad X_0(x)=x
\end{equation}
where $W$ is a Brownian motion;  $\nabla^T X^{-1}$ in \eqref{eq:constantin_iyer} denotes the transpose Jacobian matrix of $X^{-1}$.

Representation \eqref{eq:constantin_iyer}-\eqref{eq:constantin_iyer2} is non-linear in the sense of McKean, as $X$ solves an equation driven by the velocity field $u$, which depends non-trivially on the law of $(X,\nabla X)$ itself. It allows to recast the PDE problem of solving the Navier--Stokes equations \eqref{eq:NS_intro} as the alternative stochastic analysis one of finding a unique fixed point for system \eqref{eq:constantin_iyer}-\eqref{eq:constantin_iyer2}. Notice that, thanks to the regularizing features of Brownian paths $W$, the SDE \eqref{eq:constantin_iyer2} can have a well-defined flow even in situations where the standard Cauchy-Lipschitz theory no longer applies. Such \emph{regularization by noise} phenomena have received considerable attention in the last twenty years, for general multidimensional SDEs
\begin{equation}\label{eq:intro_SDE}
	\dd X_t = b_t(X_t)\dd t + \sqrt{2\eps} \dd W_t
\end{equation}
on $\R^d$ and any $\eps>0$. Let us only mention some representative contributions in this extremely fascinating field and refer to \cite{Flandoli2015} for a broader overview.
The seminal work of Krylov and R\"ockner \cite{KryRoc2005} first established strong existence and uniqueness of solutions to \eqref{eq:intro_SDE} under the strict Ladyzhenskaya--Prodi--Serrin (LPS) condition (say for $d\geq 2$)
\begin{align*}
	b\in L^q([0,T];L^p),\quad \frac{2}{q}+\frac{d}{p}<1.
\end{align*}
Under the same condition, \cite{FedFla2013} later established the $\alpha$-H\"older continuity of the stochastic flow $x\mapsto X_t(x)$ for any $\alpha<1$, as well as its weak differentiability with $\nabla X\in L^2(\Omega\times[0,T]\times\R^d)$; in fact, $\PP$-a.s. $\nabla X\in C([0,T];L^p)$ with suitable moment bounds for any $p<\infty$, see \cite{XXZZ2020}.
Reaching the critical equality $2/q+d/p=1$ was only accomplished recently \cite{RocZha2021}; as an application, the authors therein recover the known LPS regularity criterion for \eqref{eq:NS_intro} solely based on SDE results and the representation \eqref{eq:constantin_iyer}-\eqref{eq:constantin_iyer2}.

Although natural from a scaling point of view, the LPS condition is too restrictive when dealing with weak solutions to \eqref{eq:NS_intro}; for instance, Leray solutions are only known to satisfy
\begin{align*}
	u\in L^q([0,T]; L^p)\quad \text{for any}\quad (q,p)\in [2,\infty]\times [2,6] \quad\text{satisfying}\quad \frac{2}{q}+\frac{3}{p}=\frac{3}{2}.
\end{align*}
In particular, as pointed out in \cite{Zhao2019}, it is currently unknown whether Leray solutions to \eqref{eq:NS_intro} satisfy the representation \eqref{eq:constantin_iyer}-\eqref{eq:constantin_iyer2}.
However, Leray solutions should not be treated as just \emph{any} velocity field $b$, given the numerous additional properties they satisfy. In this regard, it was realized in \cite{ZhaZha2021} that already their divergence-free property allows to considerably relax the LPS condition: wellposedness in law for \eqref{eq:intro_SDE} holds for Lebesgue a.e. $x\in\R^d$ as soon as
\begin{equation}\label{eq:LPS_v2}
	b\in L^q([0,T]; L^p), \quad \nabla\cdot b=0 \quad \text{with}\quad (q,p)\in [2,\infty]^2 \quad\text{satisfying}\quad \frac{2}{q}+\frac{3}{p}\leq 2;
\end{equation}
in particular, condition $\nabla\cdot b=0$ allows to go beyond the scaling critical regularity and successfully cover the Leray class. See \cite{Zhao2019,hao2023} for further results in this exciting recent direction.
There is unfortunately one major shortcoming in \cite{ZhaZha2021}, as \eqref{eq:LPS_v2} only guarantees \emph{weak existence and uniqueness in law}; solutions $X^i$ associated to different initial data $x^i$ may live on different probability spaces, making it impossible to talk about a stochastic flow $x\mapsto X(x)$ and study its differentiability.

On the other hand, if one is interested in deterministic Lagrangian trajectories, i.e. solve the deterministic ODE (in integral form)
\begin{equation}\label{eq:intro_ODE}
	y_t = x + \int_0^t u_s(y_s) \dd s \quad\forall\, t\geq 0
\end{equation}
for a given Leray solution $u$, more specific results are available in the analysis literature. Building on the a priori estimates previously established in \cite{FGT1981}, Foias, Guillopé and Temam \cite{FGT1985} showed that the Cauchy problem \eqref{eq:intro_ODE} admits at least one solution for every $x$ and that moreover one can construct a measurable selection $(t,x)\mapsto y_t(x)$ such that $y_t(\cdot)$ is Lebesgue measure-preserving for every $t\geq 0$; uniqueness of such selections was however left open. This was before the advent of the theory of Regular Lagrangian Flows (RLFs) by DiPerna--Lions \cite{DiPLio1989} and Ambrosio \cite{Ambrosio2004}, which applies to Leray solution since they are divergence-free and enjoy Sobolev regularity $L^2([0,T];H^1)$.

Concerning more classical uniqueness results for the Cauchy problem, Robinson and Sadowski \cite{RobSad2009,RobSad2009b} were able to prove that existence and uniqueness for \eqref{eq:intro_ODE} for Lebesgue a.e. initial $x\in\R^d$, under the assumption that $u$ is a Leray solution with $u_0\in H^{1/2}$. Their proof is very elegant, mainly relying on a result on the upper-box counting dimension of the singular set of $u$, in the style of Caffarelli--Kohn--Niremberg partial regularity theory \cite{CKN1982}, and an avoidance result for measure-preserving flows à la Aizenman \cite{aizenman1978}.
Exactly because of its specificity however, the result does not easily generalize to other settings, already by considering the stochastic characteristics \eqref{eq:constantin_iyer2}; indeed, \cite[Proposition 3.1]{RobSad2009b} crucially relies on solutions being absolutely continuous curves, which is no longer true in the presence of $W$.

Still, the result from \cite{RobSad2009b} is rather striking as it is not at all a simple byproduct of the theory of RLFs. The latter only guarantees existence and uniqueness within a special physical class of flows $x\mapsto X_t(x)$, which is not a priori in contradiction with the existence of infinitely many solutions to \eqref{eq:intro_ODE}.
This is indeed the case: on $\R^d$ with $d\geq 2$, it is possible to construct divergence-free, Sobolev regular velocity fields $b\in C^0([0,T];W^{1,p})$ with $p<d$ such that the corresponding ODE admits non-unique solutions for initial data $x$ belonging to a set of positive Lebesgue measure. The result has been first established in \cite{BCDL2021}, relying on convex integration techniques and Ambrosio's superposition principle, and refined in \cite{Sorella2023}; a more explicit, Cantor-type construction of the velocity field $b$ has then been presented in \cite{Kumar2023}.
In contrast, in the regularity regime $b\in L^1([0,T];W^{1,p})$ with $p>d$, trajectorial a.e. uniqueness of the integral curves holds, as shown in \cite{CarCri2021}; the critical case $p=d$ as then been analyzed in the case of borderline Lorentz regularity $\nabla b\in L^1([0,T]L^{d,1})$ in \cite{BCDL2021}.
Both results fundamentally leverage on the fact that, in this regularity regime, $b$ satisfies an \emph{asymmetric Lusin--Lipschitz} property: there exists another locally integrable function $g$ such that
\begin{equation}\label{eq:asymmetric_intro}
	|b_t(x)-b_t(y)|\leq g_t(x) |x-y|\quad \text{for Lebesgue a.e. } (t,x,y)
\end{equation}
which is considerably stronger than the standard version of the same property (where the symmetric weight $g_t(x)+g_t(y)$ would appear).

The goal of the present work is to revisit the Robinson--Sadowski result on a.e. uniqueness for \eqref{eq:intro_ODE} through the lens of the more general, functional-analytic framework developed in the context of Regular Lagrangian flows and in particular the works \cite{CarCri2021,BCDL2021}. This allows to strengthen the result and treat the SDE \eqref{eq:constantin_iyer2} as well. 

\subsection{Main results and ideas of proof}\label{subsec:main.results}

Throughout the paper, when discussing $3$D Navier--Stokes we will always restrict ourselves to Leray solutions, as defined below. In the next definition, $L^2_\sigma$ is the subspace of $L^2(\R^3;\R^3)$ consisting of divergence-free vector fields, in the sense of distributions.

\begin{definition}\label{defn:leray_solution}
Given $u_0\in L^2_\sigma$, we say that $u$ is a Leray solution to Navier--Stokes equations \eqref{eq:NS_intro} if there exists a sequence of standard mollifiers $(\rho^n)_{n\in\N}$, and a sequence $\{u_0^n\}_{n\in\N}\subset L^2_\sigma\cap C^\infty$ such that $u_0^n\to u_0$ in $L^2$ and $u^n\to u$ in $L^2_\loc([0,+\infty);L^2)$, where $u^n$ are the solutions to the regularised systems
\begin{equation}\label{eq:leray_scheme}\begin{cases}
	\partial_t u^n + [(\rho^n\ast u^n)\cdot\nabla] u^n = \nu\Delta u^n + \nabla p^n\\
	\nabla\cdot u^n = 0, \quad u^n\vert_{t=0}=u^n_0.
\end{cases}\end{equation}
\end{definition}

For any $u_0\in L^2_\sigma$, there exist solutions $u$ to Navier--Stokes in the sense of the above definition, which is faithful to Leray's original construction; see e.g. \cite[Chapter 14]{RRS2016}. Definition \ref{defn:leray_solution} is not the only possible way to define Leray solutions, with several alternatives in the literature; often, authors just require $u$ to satisfy \eqref{eq:NS_intro} in the sense of distribution, as well as some additional physical properties, like the $L^\infty([0,+\infty);L^2)\cap L^2([0,+\infty);\dot H^1)$-bound coming from the energy inequality; see e.g. the discussion around \cite[Defn. 4.1]{BedVic2022}. 

Definition \ref{defn:leray_solution} is in a sense the strongest possible one, as it implies that all uniform-in-$n$ a priori bounds for $u^n$ solving \eqref{eq:leray_scheme} transfer to the limit solution (possibly in the form of an inequality, by weak lower-semicontinuity).
For instance, Leray solutions in the sense of Definition \ref{defn:leray_solution} must satisfy the \emph{strong energy inequality}
\begin{equation}\label{eq:leray_strong_energy}
	\| u_t\|_{L^2}^2 + 2\nu\int_s^t \| \nabla u_r\|_{L^2}^2 \dd r \leq \| u_s\|_{L^2}^2 \quad \forall\, t\geq s
\end{equation}
valid for Lebesgue a.e. $s\in [0,+\infty)$ and in particular for $s=0$; they are admissible in the sense of Caffarelli--Kohn--Niremberg \cite{CKN1982} and satisfy the weak-strong uniqueness principle. Solutions are weakly continuous in time, $u_t\to u_0$ in $L^2$ as $t\to 0^+$ and satisfy $u\in L^\infty([0,+\infty); L^2)$, $\nabla u\in L^2([0,+\infty); L^2)$ and $\partial_t u\in L^{4/3}([0,+\infty);\dot H^{-1})$, with bounds only depending on $\| u_0\|_{L^2}$.
Finally, similarly to the strong energy inequality, there exists a full Lebesgue measure set $\Gamma\subset [0,+\infty)$ such that, for any $s\in\Gamma$, $u_{s+\,\cdot}$ is still a Leray solution starting from $u_s\in L^2$, in the sense of the above definition.

From now on, for notational simplicity we will always take $\nu=1$. All results however hold without modifications for any $\nu>0$. Our first main statement is the following.

\begin{theorem}\label{thm_main1_intro}
Let $u_0\in L^2_\sigma$ and $u$ be an associated Leray solution. Then the ODE problem
\begin{equation*}
	y_t = x + \int_0^t u_s(y_s) \dd s \quad\forall\, t\geq 0
\end{equation*}
is well-defined and for Lebesgue a.e. $x\in\R^3$ there exists exactly one solution to the above ODE.
\end{theorem}

Compared to the original result from \cite{RobSad2009,RobSad2009b}, the assumption $u_0\in H^{1/2}$ is no longer required in Theorem \ref{thm_main1_intro}.

As mentioned, our result relies on the RLF framework and in particular the asymmetric estimate \eqref{eq:asymmetric_intro}. By combining the a priori estimates from \cite{FGT1981} (Lemma \ref{lem:FGT} here) with interpolation-type estimates in Lorentz spaces (Corollary \ref{cor:refined_inequality}), we are able to show that any Leray solution $u$ satisfies $\nabla u\in L^1_\loc([0,+\infty);L^{3,1})$; using the arguments from \cite{BCDL2021} (cf. Corollary \ref{cor:asymmetric_lusin_lipschitz_flow}), we can conclude that $u$ satisfies \eqref{eq:asymmetric_intro} with $g\in L^1_\loc([0,+\infty);L^{3,\infty})$, see Corollary \ref{cor:onesided_navier_stokes}. Once property \eqref{eq:asymmetric_intro} is verified, trajectorial a.e. uniqueness follows from a general driving principle within the RLF framework, cf. Proposition \ref{prop:sufficient_ae_uniqueness}.

Our result is in fact more general than Theorem \ref{thm_main1_intro}, as it allows for the presence of any given continuous forcing $\gamma\in C([0,T];\R^3)$ on the r.h.s. of \eqref{eq:intro_ODE}, see Theorem \ref{thm:main1} for the precise statement; in this case, the negligible set of initial conditions may depend on $\gamma$. As a natural byproduct, one can deduce a.e. uniqueness statements in the case where $\gamma$ is sampled randomly, for instance (but not necessarily) according to the Wiener measure. This is quite convenient in order to tackle the SDE \eqref{eq:constantin_iyer2}.

For this reason, as a preliminary result, we extend the current theory of RLFS to accommodate for the presence of forcing $\gamma$; the resulting flow $X$ depends continuously on both the Sobolev velocity field $b$ and $\gamma$, see  Theorem \ref{thm:existence_RLF_gamma}.
In the context of Brownian noise $W$, such kind of results are not new and extensions of the classical RLF theory have been proposed by several authors, in particular Le Bris--Lions \cite{LeBLio2004,LeBLio2008}, Figalli \cite{Fig2008} and Zhang \cite{Zhang2010}; for a more complete account, we refer to \cite{ZhaZha2021,Zhao2019} and the references therein.
Our approach is more pathwise in nature, as carried out for any fixed $\gamma$ by mostly readapting estimates à la Crippa--De Lellis \cite{CriDeL2008}. We find it more elastic, as it allows to sample $\gamma$ from very different probability measures on $C([0,T];\R^d)$, possibly not related to Markovian or martingale processes. In particular, Fokker--Planck PDEs or martingale problems never appear in our setting.

As a result of independent interest, within this extended framework, we study the question of convergence of Picard iterations for the ODE \eqref{eq:intro_ODE} under the assumptions guaranteeing a.e. uniqueness, see Section \ref{subsec:picard}. This suggests the possibility of numerically investigating the Lagrangian trajectories \eqref{eq:intro_ODE}, even for drifts of very poor regularity.

In the case of stochastic Lagrangian trajectories associated to $u$, employing similar arguments as for Theorem \ref{thm_main1_intro} yields strong existence and path-by-path uniqueness of solutions for Lebesgue a.e. $x\in\R^3$.
Under additional regularity on $u_0$, we can further exploit the nondegeneracy and diffusive nature of the noise, to strengthen the result by covering every initial condition $x\in\R^3$.

\begin{theorem}\label{thm_main2_intro}
Let $u_0\in L^2_\sigma$, $u$ be an associated Leray solution and let $\eps>0$ fixed.
Then for Lebesgue a.e. $x\in\R^3$, strong existence and pathwise uniqueness holds for the SDE problem
\begin{equation}\label{eq:intro_SDE_NS}
	Y_t = x + \int_0^t u_s(Y_s) \dd s + \eps W_t \quad\forall\, t\geq 0
\end{equation}
where $W$ denotes Brownian motion.

If additionally $u_0\in H^{1/2}$, then the same conclusion holds for \emph{every} initial condition $x\in\R^3$.
\end{theorem}

We actually obtain a slightly stronger result, concerning \emph{path-by-path uniqueness}, see Theorems \ref{thm:main.intermediate}-\ref{thm:main2} for the precise statement. Roughly speaking, this means that for a.e. fixed realization $W(\omega)$ of the noise, uniqueness holds among all possible integral curves for \eqref{eq:intro_SDE_NS}, not necessarily adapted to a reference filtration; such property of uniqueness was first established in the context of singular SDEs by Davie \cite{Davie2007}, see \cite{Flandoli2009} for a deeper discussion.

The second part of Theorem \ref{thm_main2_intro} additionally requires $u_0\in H^{1/2}$, for technical reasons similar to those from \cite{RobSad2009b}. Indeed, as this space is critical for Navier--Stokes, it implies the existence of a strong, $H^{1/2}$-valued solution $u$ on a short time $[0,T^\ast]$, which can be then extended to $[0,+\infty)$ as a Leray one; in particular, it prevents the possibility of trajectories immediately splitting at $t=0$. Notice however that the uniqueness result is then subsequently true at all positive times, even after the solution $u$ has possibly developed blow-up in $H^{1/2}$. This is because the presence of $W$ has already helped the solution becoming diffuse  at $t>0$, a property which is then preserved by the dynamics (see the proof of Theorem \ref{thm:main2} for the exact details).
We expect the same result to be true for $u_0$ belonging to other critical classes (cf. Remark \ref{rem:initial data}) and conjecture such conditions to be unnecessary (cf. Conjecture \ref{conjecture}).
In any case, even with this restriction on $u_0$, to the best of our knowledge this is the first result concerning global well-posedness of the stochastic characteristics \eqref{eq:intro_SDE_NS} associated to Leray solutions to $3$D Navier--Stokes.

In Section \ref{subsec:consequences} we collect a few consequence of our main result, concerning the properties of the two-parameter semigroup associated to the SDE, namely $P_{s\to t}\varphi(x)=\E[\varphi(X_{s\to t}(x))]$.
It should be noted that our well-posedness result so far does not imply any H\"older regularity for the stochastic flow $x\mapsto X_t(x)$, not to mention moment estimates for $|\nabla X_t(x)|$. In this sense, we are still very far from answering the question whether one can give meaning to the Constantin--Iyer representation \eqref{eq:constantin_iyer}-\eqref{eq:constantin_iyer2} for Leray solutions. Current available bounds for $\nabla X$ are rather logarithmic in nature and come from pathwise estimates for DiPerna--Lions flows, cf. Remark \ref{rem:regularity_gradient}. We leave this challenging question for future investigations.

\subsection*{Structure of the paper}
We conclude this introduction by collecting the main notations and conventions we will use.
Section \ref{sec:flows} presents a pathwise theory of random Regular Lagrangian Flows (constructed in Sec. \ref{subsec:RLF}) and discusses the connection between the asymmetric Lusin--Lipschitz property and trajectorial a.e. uniqueness results (Sec. \ref{subsec:asymmetric_lusin}), also in relation to convergence of Picard iterations (Sec. \ref{subsec:picard}).
Section \ref{sec:interpolation} is devoted to some functional inequalities in Lorentz and Besov spaces (respectively Sec. \ref{subsec:lorentz} and \ref{subsec:besov}) needed to establish asymmetric inequalities.
Finally, Section \ref{sec:main-proof} is devoted to the proof of the asymmetric Lusin--Lipschitz estimate for Leray solutions (Sec. \ref{subsec:regularity_leray}) and of the main results (Sec. \ref{subsec:proofs_main}); we shortly collect some consequences in Sec. \ref{subsec:consequences}.

\subsection*{Notations and conventions}
We write $a\lesssim b$ to mean that there exists a positive constant $C$ such that $a \leq C b$; we use the index $a\lesssim_{\alpha,\beta} b$ to highlight the dependence of the hidden constant $C$ on some relevant parameters $\alpha,\beta$.

For any $d,m\in \mathbb{N}$ and $p\in [1,\infty]$, we denote by $L^p(\R^d;\R^m)$ the standard Lebesgue space; when there is no risk of confusion in the parameters $d,m$, we will simply write $L^p$ for short and denote by $\| \cdot\|_{L^p}$ the corresponding norm. Analogously, we denote by $L^p_{loc}(\R^d;\R^m)=L^p_{loc}$ local Lebesgue spaces, endowed with their natural Frechét topology. Similar conventions applies to Sobolev spaces $W^{1,p}=W^{1,p}(\R^d;\R^m)$, $W^{1,p}_{loc}$ and so on.

The Lebesgue measure on $\R^d$ is denoted by $\mathscr{L}^d$.
$C^0=C^0(\R^d;\R^m)$ stands for the Banach space of continuous, bounded functions endowed with supremum norm.
Throughout the paper, several scales of spaces will be used, like Lorentz $L^{p,q}$ or Besov $B^s_{p,q}$; we refer to Sections \ref{subsec:lorentz}-\ref{subsec:besov} for their exact definitions. $\dot H^s=\dot H^s(\R^d;\R^m)$ denote homogenous fractional Sobolev spaces, namely the collection of all tempered distributions $f$ whose Fourier transform $\hat f$ satisfies $\| f\|_{\dot H^s}^2 = \int_{\R^d} |\xi|^{2s} |\hat f(\xi)|^2 \dd \xi<\infty$.

For convenience, most of the time we will work on a fixed finite time interval $[0,T]$, although arbitrarily large. In this case, we adopt the shortcut $\cC_T$ for the path space $C([0,T];\R^d)$ and denote by $\gamma$ its elements. Given a Banach space $E$ and $q\in [1,\infty]$, we denote by $L^q_T E=L^q([0,T];E)$ the Lebesgue--Bochner space, consisting of strongly measurable functions $f:[0,T]\to E$ such that $\int_0^T \| f_s\|_E^q \dd s<\infty$ (with usual modifications for $q=\infty$). This notation can be concatenated with the previous ones, allowing to define e.g. $L^1_T C^0$, $L^q_T W^{1,p}$, etc.

Given a Banach space $E$, we denote by $\cP(E)$ the set of probability measures on $E$ (w.r.t. Borel $\sigma$-algebra). Whenever considering a stochastic process $W$, if not already specified we will be implicitly assuming the existence of an underlying probability space $(\Omega,\cF,\PP)$ where it is defined. If additionally the space if endowed with a filtration $\{\cF_t\}_t$, then $\cF$ and $\{\cF_t\}_t$ are assumed to satisfy the standard properties (completeness, right continuity).

\section{(Random) Regular Lagrangian flows and a.e. uniqueness}\label{sec:flows}

In this section, we lay out in Section \ref{subsec:RLF} a theory of (Random) Regular Lagrangian Flows for the ODE
\begin{equation*}
	x_t = x + \int_0^t b_s(x_s) \dd s + \gamma_t.
\end{equation*}
where the forcing term $\gamma$ can be sampled randomly. Within this framework, we provide pathwise criteria for a.e. trajectorial uniqueness in Section \ref{subsec:asymmetric_lusin} and derive some interesting consequences, including convergence of Picard iterations for the ODE in Section \ref{subsec:picard}.

\subsection{(Random) Regular Lagrangian flows}\label{subsec:RLF}

We start by providing a notion of Regular Lagrangian Flow (RFL) for an ODE with a given continuous forcing $\gamma$ on the r.h.s. Such a concept already appeared before in the literature, cf. \cite[Sec. 5.1]{LeBLio2004}, although under the terminology of ``a.e. flow''.

\begin{definition}\label{defn:RLF_gamma}
Let $b\in L^1([0,T]\times \R^d;\R^d)$, $\gamma\in C([0,T];\R^d)$ and consider the integral ODE
\begin{equation}\label{eq:ode_gamma_RLF}
	x_t = x + \int_0^t b_s(x_s) \dd s + \gamma_t.
\end{equation}
We say that $X:[0,T]\times\R^d\to\R^d$ is a \emph{Regular Lagrangian Flow} for \eqref{eq:ode_gamma_RLF} if:
\begin{itemize}
\item[i)] For a.e. $x\in\R^d$, $\int_0^T |b_s(X_s(x))| \dd s<\infty$ and $Y_t(x):= X_t(x)-\gamma_t$ is an absolutely continuous curve satisfying $Y_0(x)=x+\gamma_0$ and $\dot Y_t(x) = b_t(X_t(x))$.
\item[ii)] There exists a constant $L$ such that
\begin{align*}
	\mathscr{L}^d(X_t^{-1}(A)) \leq L\, \mathscr{L}^d(A)
\end{align*}
for every Borel set $A\subset \R^d$ and every $t\in [0,T]$.
\end{itemize}
\end{definition}

We will soon present a result guaranteeing existence and uniqueness of RLFs for a suitable class of drifts $b$, which will be sufficient for our purposes.
To this end, we need to introduce some notations.
We always work on $d\geq 2$. Given $T\in (0,+\infty)$ and $p\in (1,\infty)$, we define
\begin{equation}\label{eq:drift_space}
	 \mathscr{W}^p_T :=\big\{ b\in L^1_T W^{1,p} : \nabla\cdot b\equiv 0\big\}
\end{equation}
which is a Banach space endowed with the $L^1_T W^{1,p}$-norm.

Inspired by \cite{Zhang2013}, we define a probability measure $\nu$ on $\R^d$ by setting $\nu(\dd x)= c_d (1+|x|)^{-d-1}$, where $c_d$ is the renormalizing constant such that $\nu(\R^d)=1$. We denote the corresponding $L^p(\R^d,\nu)$ spaces by $L^p_\nu$ for short.
They serve the purpose of metrizing convergence in $L^p_\loc$: for a bounded sequence $\{f^n\}_n\subset L^p$, it holds
\begin{align*}
	f^n\to f \text{ in } L^p_\loc \quad \Leftrightarrow \quad \lim_{n\to\infty} \| f_n-f\|_{L^p_\nu}=0.
\end{align*}
Further notice that by construction $\| f\|_{L^p_\nu} \lesssim_d \| f\|_{L^p}$ for all $p\in [1,\infty]$.
Similarly, we denote by $L^p_\nu(\cC_T)$ the Lebesgue--Bochner space $L^p(\R^d,\nu;\cC_T)$.

\begin{theorem}\label{thm:existence_RLF_gamma}
	Let $d\geq 2$, $T\in (0,+\infty)$, $p\in (1,\infty)$. Then for any $(b,\gamma)\in\mathscr{W}^p_T\times \cC_T$ there exists a unique Regular Lagrangian Flow associated to \eqref{eq:ode_gamma_RLF}, in the sense of Definition \ref{defn:RLF_gamma}, which moreover is incompressible, in the sense that it leaves the Lebesgue measure $\mathscr{L}^d$ invariant.
	
	Moreover, there exists a constant $C=C(d,p)$ such that, if $X^i$ are the RLFs associated to distinct $(b^i,\gamma^i)\in\mathscr{W}^p_T\times \cC_T$ for $i=1,2$, then  
	\begin{equation}\label{eq:stability_RLF}
		\big\| 1\wedge \| X^1-X^2\|_{\cC_T}\big\|_{L^p_\nu}
		\leq C \Big[ e^\lambda \big( \| b^1-b^2\|_{L^1_T L^p} + \| \gamma^1-\gamma^2\|_{\cC_T} \big) + \frac{1}{\lambda} \| \nabla b^1\|_{L^1_T L^p} \Big]
	\end{equation}
	for all $\lambda>0$.
	If $\{(b^n,\gamma^n)\}_n\subset \mathscr{W}^p_T\times \cC_T$ and $(b,\gamma)\in\mathscr{W}^p_T\times \cC_T$ are such that
	\begin{equation}\label{eq:convergence_RLF_assumption}
		\lim_{n\to\infty} \|b^n-b\|_{L^1_T L^p}=0, \quad \lim_{n\to\infty} \|\gamma^n-\gamma\|_{\cC_T}=0,
	\end{equation}
	then for the corresponding RLFs $X^n$ (resp. $X$) associated to $(b^n,\gamma^n)$ (resp. $(b,\gamma)$) it holds that
	\begin{equation}\label{eq:convergence_RLF}
		\lim_{n\to\infty} \int_{\R^d} \| X^n(x)-X(x)\|_{\cC_T}^p \nu(\dd x)=0.
	\end{equation}
	In particular, the map $(b,\gamma)\mapsto \Phi_t(x;b,\gamma):=X_t(x)$ is continuous from $\mathscr{W}^p_T \times \cC_T$ to $L^p_\nu(\cC_T)$.
\end{theorem}

\begin{proof}
	Existence and uniqueness of the RLF follows by the same Galilean transformation implemented in \cite[Sec. 5.1]{LeBLio2004}: $X$ is an RLF associated to $(b,\gamma)$ if and only if $Y:=X-\gamma$ is an RLF associated to $(b^\gamma,0)$, where $b^\gamma_t(z):= b_t(z+\gamma_t)$. Since $b^\gamma\in \mathscr{W}^p_T$ with $\| b^\gamma\|_{\mathscr{W}^p_T} = \| b\|_{\mathscr{W}^p_T}$, we can invoke the classical results from \cite{DiPLio1989,Ambrosio2004,CriDeL2008} to deduce existence and uniqueness for $Y$, implying the same statement for $X$. Similarly, the incompressibility of $X$ follows from $\nabla \cdot b^\gamma\equiv 0$, the same property for $Y$ and the fact the fact that it is preserved by translations $z\mapsto z+\gamma_t$.
	
	Estimate \eqref{eq:stability_RLF} is based on a variant of the arguments from \cite{CriDeL2008}; for simplicity, we present the argument for smooth $b^i$, the general case then following by standard density arguments.

	For any $x\in \R^d$, let $X^i(x)$ be solutions associated to $(b^i,\gamma^i)$, then for any $t\in [0,T]$ it holds
	\begin{equation}\label{eq:stability_RLF_eq0}\begin{split}
		|X^1_t(x)-X^2_t(x)|
		& \leq \int_0^t |b^1_s(X^1_s(x)) - b^1_s(X^2_s(x))| \dd s\\
		& \quad + \int_0^T |(b^1_s-b^2_s)(X^2_s(x))| \dd s + \| \gamma^1-\gamma^2\|_{\cC_T}.
	\end{split}\end{equation}
	Set $B_T(x):=\int_0^T |(b^1_s-b^2_s)(X^2_s(x))| \dd s$; since $X^2_s$ leaves $\mathscr{L}^d$ invariant, by Minkowki's inequality it holds
	\begin{align*}
		\| B_T\|_{L^p} \leq \int_0^T \| b^1_s-b^2_s\|_{L^p} \dd s = \| b^1-b^2\|_{L^1_T L^p}.
	\end{align*}
	On the other hand, by Hajlasz's inequality \cite[Lemma A.3]{CriDeL2008}, setting
	\begin{align*}
		h_t(x):= \mathcal{M} \nabla b_t(X^1_t(x)) + \mathcal{M} \nabla b_t(X^2_t(x)),
	\end{align*}		
	where $\mathcal{M}$ denotes the Hardy--Littlewood maximal function, we have
	\begin{align*}
		\int_0^t |b^1_s(X^1_s(x)) - b^1_s(X^2_s(x))| \dd s
		\leq c_d \int_0^t h_t(x) |X^1_s(x) - X^2_s(x)| \dd s.
	\end{align*}
	Inserting all these estimates in \eqref{eq:stability_RLF_eq0} and applying Gr\"onwall's lemma, we arrive at
	\begin{equation}\label{eq:RLF_key_estim}
		\| X^1(x)-X^2(x)\|_{\cC_T} = \sup_{t\in [0,T]} |X^1_t(x)-X^2_t(x)| \leq e^{H_T(x)} (B_T(x) + \| \gamma_1-\gamma_2\|_{\cC_T})
	\end{equation}		
	where $H_T(x):=c_d\int_0^T h_t(x) \dd t$. By the same argument used for $B$, we have
	\begin{align*}
		\| H_T\|_{L^p} \leq c_d \int_0^T \| h_t\|_{L^p}
		\leq 2 c_d \int_0^T \| \mathcal{M} \nabla b^1_t\|_{L^p} \dd t
		\lesssim_{p,d} \int_0^T \| \nabla b^1_t\|_{L^p} \dd t
	\end{align*}
	where in the last step we used the fact that $\mathcal{M}$ is a bounded operator on $L^p$, for $p\in (1,\infty)$.
	For any $\lambda>0$, set $E_\lambda :=\{ x\in\R^d: H_T(x)\leq \lambda\}$; then by Minkowski's inequality and \eqref{eq:RLF_key_estim}, it holds
	\begin{align*}
		\big\| 1\wedge \| X^1-X^2\|_{\cC_T}\big\|_{L^p_\nu}
		& \leq \big\| \big(1\wedge \| X^1-X^2\|_{\cC_T}\big) \mathbbm{1}_{E_\lambda} \big\|_{L^p_\nu} + \big\| \big(1\wedge \| X^1-X^2\|_{\cC_T}\big) \mathbbm{1}_{E_\lambda^c} \big\|_{L^p_\nu}\\
		& \leq \big\| \| X^1-X^2\|_{\cC_T} \mathbbm{1}_{E_\lambda} \big\|_{L^p_\nu} + \| \mathbbm{1}_{E_\lambda^c} \|_{L^p_\nu}\\
		& \lesssim_d e^\lambda \big( \| B_T\|_{L^p_\nu} + \| \gamma^1-\gamma^2\|_{\cC_T} \big) + \| \mathbbm{1}_{E_\lambda^c} \|_{L^p}\\
		& \lesssim_d e^\lambda \big( \| B_T\|_{L^p} + \| \gamma^1-\gamma^2\|_{\cC_T} \big) +\frac{1}{\lambda} \| H_T\|_{L^p} \\
		& \lesssim_{p,d} e^\lambda \big( \| b^1-b^2\|_{L^1_T L^p} + \| \gamma^1-\gamma^2\|_{\cC_T} \big) + \frac{1}{\lambda} \| \nabla b^1\|_{L^1_T L^p}
	\end{align*}
	yielding \eqref{eq:stability_RLF}.
	
	Now let $(b^n,\gamma^n)$ be a sequence satisfying \eqref{eq:convergence_RLF_assumption}, then by passing to the limit as $n\to\infty$ in \eqref{eq:stability_RLF} we find
	\begin{align*}
		\limsup_{n\to\infty} \big\| 1 \wedge \| X^n-X\|_{\cC_T} \big\|_{L^p_\nu} \leq \frac{C}{\lambda} \| \nabla b^1\|_{L^1_T L^p};
	\end{align*}
	by the arbitrariness of $\lambda$, we deduce that the l.h.s. is $0$ and so that $X^n$ converge in measure to $X$ (w.r.t. $\nu$). To deduce convergence in $L^p_\nu$, notice that similarly to \eqref{eq:stability_RLF_eq0}, we have
	\begin{align*}
		\| X^n(x)-X(x)\|_{\cC_T} \leq \int_0^T |b^n_s(X^n_s(x)) - b_s(X_s(x))| \dd s + B^n_T(x) + \| \gamma^n-\gamma\|_{\cC_T}
	\end{align*}
	for $B^n:=\int_0^T |(b^n_s-b_s)(X^n_s(x))| \dd s$.
	Set $I^n_T(x):= \int_0^T |b_s(X^n_s(x)) - b_s(X_s(x))| \dd s$; taking $L^p_\nu$-norms on both sides, arguing similarly to above, we arrive at
	\begin{align*}
		\big\| \| X^n-X\|_{\cC_T} \big\|_{L^p_\nu} \leq \| I^n_T\|_{L^p_\nu} + \| b^n-b\|_{L^1_T L^p} + \| \gamma^n-\gamma\|_{\cC_T}.
	\end{align*}
	By assumption, the last two terms on the r.h.s. converge to $0$, so it remains to show that $\| I^n_T\|_{L^p_\nu}\to 0$ as well. Since $b\in L^1([0,T];L^p)$, for any $\eps>0$ we can find $\varphi\in C^\infty_c([0,T]\times\R^d)$ such that $\| b-\varphi\|_{L^1_T L^p}<\eps$; since $X^n$ and $X$ are both incompressible, it follows that
	\begin{align*}
		\| I^n_T\|_{L^p_\nu}
		& \leq \int_0^T \| \varphi_s(X^n_s(x))-\varphi_s(X_s(x))\|_{L^p_\nu} \dd s + \int_0^T \|(b_s-\varphi_s)(X^n_s)\|_{L^p_\nu} \dd s + \int_0^T \|(b_s-\varphi_s)(X_s)\|_{L^p_\nu} \dd s\\
		& \lesssim \int_0^T \| \varphi_s(X^n_s(x))-\varphi_s(X_s(x))\|_{L^p_\nu} \dd s + \varepsilon.
	\end{align*}
	Since $\varphi$ is uniformly bounded and $X^n\to X$ in measure w.r.t. $\nu$, by dominated convergence we conclude that
	\begin{align*}
		\limsup_{n\to\infty} \| I^n_T\|_{L^p_\nu} \lesssim \varepsilon;
	\end{align*}
	by the arbitrariness of $\varepsilon>0$, conclusion \eqref{eq:convergence_RLF} follows.
\end{proof}

\begin{remark}
	In the case $\gamma^n=\gamma$ for all $n\in\N$, by refining the argument, one can further improve \eqref{eq:convergence_RLF} to show that
	\begin{align*}
		\lim_{n\to\infty} \int_{\R^d} \| X^n(x)-X(x)\|_{\cC_T}^p \dd x = 0.
	\end{align*}
	The same is not true for $\gamma^n\neq \gamma$, as one can see by simply taking $b^n=b=0$.
	Let us stress that the continuity property \eqref{eq:convergence_RLF_assumption}-\eqref{eq:convergence_RLF}, although well-known, is very strong: the sequence $\{b^n\}_n$ needs to satisfy no boundedness assumption whatsoever in $\mathscr{W}^p_T$.
	On the other hand, for fixed $b\in\mathscr{W}^p_T$, as a consequence of Theorem \ref{thm:existence_RLF_gamma}, the solution map $\gamma\mapsto \Phi(\cdot\,;b,\gamma)$ from $\cC_T$ to $L^p_\nu \cC_T$ is continuous and \textit{adapted}.
	By the latter we mean that $\Phi(\cdot\,;b,\gamma^1)\vert_{s\in [0,t]} = \Phi(\cdot\,;b,\gamma^2)\vert_{s\in [0,t]}$ whenever $\gamma^1\vert_{s\in [0,t]}=\gamma^2\vert_{s\in [0,t]}$.
\end{remark}

%
%
	%
	%

	

\begin{remark}
The pair $(b,\gamma)\in \mathscr{W}^p_T\times \cC_T$ actually induces a two-parameter flow $X_{s\to t}(x)$, which solves system of ODEs
\begin{equation*}
	X_{s\to t}(x) = x + \int_s^t b_r(X_{s\to r}(x)) \dd r + \gamma_t-\gamma_s.
\end{equation*}
Similar estimates as in Theorem \ref{thm:existence_RLF_gamma} can be deduced more generally for $X_{s\to t}(x)$, and by time-reversal also for its inverse map $X_{s\leftarrow t}(x)$. We avoided presenting the results in this generality not too make the notation too burdensome.
\end{remark}

\begin{remark}
	Theorem \ref{thm:existence_RLF_gamma} assumes for simplicity $b\in \mathscr{W}^p_T$, but several generalizations are possible. Following \cite{DiPLio1989,CriDeL2008}, to get quantitative stability estimates for the RLF it suffices to require
	\begin{align*}
		\nabla \cdot b\in L^1_T L^\infty, \quad
		\frac{b}{1+|x|} \in L^1_T L^1 + L^1_T L^\infty, \quad
		\nabla b\in L^1_T L^p_\loc \text{ for some } p\in (1,\infty). 
	\end{align*}
We avoided such generality, since endowing the resulting collection of drifts $b$ with a topology is rather technical, compared to $\mathscr{W}^p_T$; for our applications, the latter is quite convenient since it is endowed with a Polish topology that makes the solution map $(b,\gamma)\mapsto \Phi(\cdot\,;b,\gamma)$ continuous.

Other possible extensions include the case when $\nabla b$ is the singular integral of an $L^1_T L^1$-function, see \cite{BouCri2013}, or the $BV$ assumption by Ambrosio \cite{Ambrosio2004}, in which case however quantitative stability estimates are not available anymore. In this direction, when $\gamma$ is sampled as Brownian motion, see the works \cite{Fig2008,AttFla2011}.
Finally, we required $\gamma\in \cC_T$ for convenience, but leveraging on the Galilean transformation $Y=X-\gamma$, this assumption could be further relaxed.
\end{remark}

Thanks to Theorem \ref{thm:existence_RLF_gamma}, we can now consider the case when $\gamma$ is sampled randomly, by fixing its realization and applying the theory at a pathwise level. We can formalize the concept as follows.

\begin{definition}
	Let $p\in (1,\infty)$, $b\in \mathscr{W}^p_T$ and let $\mu$ be a probability measure on $\cC_T$. The Random Regular Lagrangian Flow $X^\mu$ associated to $(b,\mu)$ is the pushforward of $\mu$ under the solution map $(b,\gamma)\mapsto \Phi(\cdot\,;b,\gamma)$. In particular, $X^\mu\in \cP( L^p_\nu (\cC_T))$.
\end{definition}

\begin{remark}\label{rem:random.representation}
	Alternatively, given a probability space $(\Omega,\cF,\PP)$ on which a process $W$ with law $\mu$ is defined, we can construct a realization of the Random RLF $X^\mu$ by solving the ODE
\begin{equation*}
	X_t(x,\omega) = x + \int_0^t b_s(X_s(x,\omega)) \dd s + W_t(\omega),
\end{equation*}
so that $X(\omega)=\Phi(\cdot\,;b,W(\omega))$.
In this way, we can either regard $X$ as a $L^p_\nu \cC_T$-random variable on $(\Omega,\cF,\PP)$, or as a continuous process $t\mapsto X_t$ taking values in $L^p_\nu$. By construction, the process $X_t$ is \emph{adapted} to the (standard augmentation of the) filtration generated by $W$.
\end{remark}

\begin{remark}
	Let $(\Omega,\cF,\PP)$ be a probability space, $W$ a continuous process on it; for each $\eps>0$, set $X^\eps:=\Phi(\cdot\,;b,\sqrt\eps W)$.
	
	If $\{\sqrt\eps W\}_{\eps>0}$ satisfies a large deviation principle (LDP) as $\eps\to 0$ (as it is always the case if $W$ is Gaussian by Schilder's theorem \cite{DemZei2010}), then by the continuity of the solution map and the contraction principle, $\{X^\eps\}_{\eps>0}$ also satisfies an LDP as $\eps\to 0$, as a family of $L^p_\nu \cC_T$-valued random variables.
	This is in line with the results from \cite{Zhang2013}, whose proofs relied on the weak convergence approach to large deviations instead.
	
	When $W$ is Brownian, the convergence $\Phi(\cdot\,;b,\sqrt\eps W)\to \Phi(\cdot\,;b,0)$ implies stability of solutions to advection-diffusion PDEs in the vanishing viscosity limit. This can be made quantitative: by choosing $\lambda$ so that $e^\lambda=\eps^{-1/4}$ in \eqref{eq:stability_RLF}, one finds	 the pathwise bound
\begin{equation*}
		\big\| 1\wedge \| X^\eps(\omega)-X\|_{\cC_T}\big\|_{L^p_\nu}
		\lesssim \eps^{1/4} \| W(\omega)\|_{\cC_T} + \frac{\| \nabla b^1\|_{L^1_T L^p}}{|\log \eps|} 
\end{equation*}
Upon taking expectation, this recovers the convergence rates from \cite{BoCiCr2022}. 
\end{remark}

\subsection{Asymmetric Lusin--Lipschitz property and a.e. trajectorial uniqueness}\label{subsec:asymmetric_lusin}

Under suitable assumptions on the drift $b$, one can upgrade results on uniqueness among RLFs to results on a.e. uniqueness for the ODE problem.

\begin{proposition}\label{prop:sufficient_ae_uniqueness}
	Suppose $b\in L^1_T L^\infty_{\loc}$ satisfies the following property: for any $R>0$ there exists $h_R\in L^1_T L^1$ such that
	\begin{equation}\label{eq:asymmetric_lusin_lipschitz}
	|b(t,x)-b(t,y)|\leq h_R(t,x) |x-y| \quad \text{for a.e. } (t,x)\in [0,T]\times B_R, \ \forall\, y\in B_R.
\end{equation}
	Let $\gamma\in \cC_T$ and further assume that there exists a regular Lagrangian flow (in the sense of Definition \ref{defn:RLF_gamma}) associated to the ODE
	\begin{equation}\label{eq:ode_gamma}
	x_t = x + \int_0^t b(s,x_s) \dd s + \gamma_t.
	\end{equation}
	Then for Lebesgue a.e. $x\in \R^d$, the Cauchy problem \eqref{eq:ode_gamma} admits exactly one solution on $[0,T]$.
\end{proposition}

\begin{proof}
	Let $\Gamma_R\subset [0,T]\times B_R$ be the Lebesgue negligible set where inequality \eqref{eq:asymmetric_lusin_lipschitz} fails to hold, $\Gamma=\cup_{R\in\N} \Gamma_R$.
	Let us define the set
	\begin{align*}
	N:=\bigg\{\,x\in\R^d: \int_0^T h_R(t,X_t(x)) \dd t <\infty \ \forall\, R\in\N, \ \| X(x)\|_{C_T} <\infty, \ \int_0^T \mathbbm{1}_\Gamma(t,X_t(x))) \dd t =0\,\bigg\}
	\end{align*}
	We claim that $N$ is a set of full Lebesgue measure. Indeed, by definition of RLF, $\| X(x)\|_{\cC_T}<\infty$ for a.e. $x$; by the quasi-incompressibility property and Fubini, it holds
	\begin{align*}
		\int_{\R^d} \int_0^T h_R(t,X_t(x)) \dd t \dd x
		\leq L\, \| h_R\|_{L^1_T L^1}<\infty, \quad
		\int_{\R^d} \int_0^T \mathbbm{1}_\Gamma(t,X_t(x))) \dd tt \dd x \leq L \| \mathbbm{1}_\Gamma\|_{L^1_T L^1} = 0.
	\end{align*}
	We now claim that, for any $x\in N$, there exists a unique solution to \eqref{eq:ode_gamma}, given by $X(x)$. Indeed, let $z_t$ be another solution to \eqref{eq:ode_gamma} and choose $R\in\N$ large enough so that $\| z\|_{C_T}, \| X_t(x)\|_{C_R} < R$. Then by property \eqref{eq:asymmetric_lusin_lipschitz} and the definition of $N$, it holds
	\begin{align*}
		|X_t-z_t| \leq \int_0^t h_R(t,X_t(x)) |X_s(x)-z_s| \dd s\quad \forall\, t\in [0,T];
	\end{align*}
	since $\int_0^T h_R(s,X_s(x)) \dd s<\infty$, by Gr\"onwall's lemma we conclude that $X(x)\equiv z$.
\end{proof}

In light of Proposition \ref{prop:sufficient_ae_uniqueness}, the problem of establishing a.e. uniqueness of trajectories reduces to the existence of a regular Lagrangian flow (already addressed in Theorem \ref{thm:existence_RLF_gamma})\footnote{Notice that, even for divergence-free fields, the existence of a measure preserving flow of solutions is non-trivial; in general it fails to hold without a weak differentiability assumption, as shown recently in \cite{pappalettera2023}.} and the validity of a fully asymmetric Lusin--Lipschitz estimate \eqref{eq:asymmetric_lusin_lipschitz}.
The starting point to derive the latter is a classical inequality due to Morrey: assuming  $b$ smooth for simplicity, setting $r=|x-y|$, it holds
\begin{equation}\label{eq:maximal_inequality}
	|b(x)-b(y)|\lesssim_d \int_{B_r(x)} \frac{|\nabla b(z)|}{|x-z|^{d-1}} \dd z + \int_{B_r(y)} \frac{|\nabla b(z)|}{|y-z|^{d-1}} \dd z \quad \forall\, x, y\in\R^d,
\end{equation}
see \cite[Lemma 3.1]{Jabin2010}; by a telescoping argument and the inclusion $B_r(y)\subset B_{2r}(x)$, \eqref{eq:maximal_inequality} implies
\begin{equation}\label{eq:maximal_inequality2}
	\frac{|b(x)-b(y)|}{|x-y|} \lesssim \cM |\nabla b|(x) + \int_{B_{2r}(x)} \frac{|\nabla b(z)|}{|y-z|^{d-1}} \dd z
\end{equation}
where $\cM$ denotes the Hardy--Littlewood maximal function operator.
Starting from \eqref{eq:maximal_inequality2}, fully asymmetric estimates have been obtained in \cite[Lemma 5.1]{CarCri2021} and \cite[Proposition 9.3]{BCDL2021}. We present below a statement which will be adapted to our setting; therein, we employ Lorentz spaces, see the upcoming Section \ref{subsec:lorentz} for their exact definition. The assumption of continuity of $b$ in the next statement comes without loss of generality, since any $b\in L^1$ with $\nabla b \in L^{d,1}$ admits a continuous representative, cf. \cite{Stein1981}. 

\begin{proposition}\label{prop:asymmetric_lorentz}
	Let $b$ be a continuous function such that $\nabla b\in L^{d,1}$. Then there exists $h\in L^{d,\infty}$ such that
	\begin{equation}\label{eq:asymmetric_lorentz}
		|b(x)-b(y)| \leq h(x) |x-y| \quad \forall\, x,y\in\R^d
	\end{equation}
	with the additional property that $\| h\|_{L^{d,\infty}} \lesssim_d \| \nabla b\|_{L^{d,1}}$. Moreover there is a continuous selection map $L^{d,1} \ni \nabla b\mapsto h\in L^{d,\infty}$.
\end{proposition}

\begin{proof}
	The statement is very close to \cite[Proposition 9.3]{BCDL2021}, treating the torus case; let us provide a short proof. By H\"older's inequality in Lorentz spaces \cite[Theorem 3.5]{ONeil1963}, for any $r>0$ it holds
	\begin{align*}
		\int_{B_{r}(x)} \frac{|\nabla b(z)|}{|y-z|^{d-1}} \dd z
		\leq \big\| \nabla b \mathbbm{1}_{B_{r}(x)} \big\|_{L^{d,1}}\, \big\| |\cdot|^{1-d}\big\|_{L^{\frac{d}{d-1},\infty}}
		\lesssim r \sup_{\delta>0} \frac{\| \nabla b \mathbbm{1}_{B_{\delta}(x)} \|_{L^{d,1}}}{\| \mathbbm{1}_{B_{\delta}(x)}\|_{L^{d,1}}}
		=: r g(x)
	\end{align*}
	where we used the fact that $\|\mathbbm{1}_{B_r(x)}\|_{L^{d,1}}= c_d\, r$ for any $r>0$. $g$ is the maximal type function considered by Stein \cite{Stein1981}, for which it holds $g\in L^{d,\infty}$ with $\| g\|_{L^{d,\infty}}\lesssim \| \nabla b\|_{L^{d,1}}$.
	Plugging these considerations in \eqref{eq:maximal_inequality2} for $r=|x-y|$, we deduce that there exists $c_d>0$ such that
	\begin{equation}\label{eq:BCD}
		|b(x)-b(y)| \leq c_d \big(\cM|\nabla b|(x) + g(x)\big) |x-y| \quad \text{ for a.e. } x,\, y\in\R^d.
	\end{equation}
	Being obtained as a maximal operator, the map $\nabla b\mapsto g$, is continuous in the right topologies; on the other hand,
	\begin{align*}
		\| \cM |\nabla b|\|_{L^{d,\infty}} \lesssim \| \cM |\nabla b|\|_{L^d} \lesssim \| \nabla b\|_{L^d} \lesssim \| \nabla b\|_{L^{d,1}}
	\end{align*}
	so the map $L^{d,1}\ni\nabla b \mapsto \cM |\nabla b|\in L^{d,\infty}$ is continuous as well.
	In this way, we deduce validity of \eqref{eq:asymmetric_lorentz} for a.e. $x,\,y\in\R^d$, upon considering $\tilde h(x):= c_d (\cM |\nabla b|(x) + g(x))$. By the continuity of $b$ and Lebesgue's differentiation theorem, we can extend the inequality to every $x,\,y\in\R^d$ upon replacing $\tilde h$ with $h(x):=\limsup_{r\to 0} \fint_{B_r(x)} \tilde h(y)\dd y$.
\end{proof}

With Proposition \ref{prop:asymmetric_lorentz} at hand, we can deduce that the (random) RLF inherits the asymmetric Lusin--Lipschitz property of the drift.  


\begin{corollary}\label{cor:asymmetric_lusin_lipschitz_flow}
	Let $d\geq 2$, $\gamma\in C([0,+\infty);\R^d)$ and assume that
	\begin{equation}\label{eq:asymmetric_lusin_lipschitz_flow_assumption}
		b\in L^1_\loc([0,+\infty); C^0),\quad \nabla \cdot b\equiv 0, \quad \nabla b\in L^1_\loc([0,+\infty);L^{d,1}).
	\end{equation}
	Then for Lebesgue a.e. $x\in \R^d$, the Cauchy problem \eqref{eq:ode_gamma} admits exactly one solution, given by $X_t(x)$.
	Moreover for any $T>0$ there exists a function $H_T$ such that the associated RLF satisfies
	\begin{equation}\label{eq:asymmetric_lusin_flow}
		\sup_{t\in [0,T]} |X_t(x)-X_t(y)| \leq e^{H_T(x)} |x-y| \quad \text{ for a.e. } x,y\in\R^d;
	\end{equation}
	$H_T\in L^{d,\infty}$ and satisfies
	\begin{align*}
		\|H_T\|_{L^{d,\infty}} \lesssim \int_0^T \| \nabla b_s\|_{L^{d,1}} \dd s.
	\end{align*}
\end{corollary}

\begin{proof}
	Since $\nabla b_t\in L^{d,1}$ for a.e. $t$, by applying Proposition \ref{prop:asymmetric_lorentz} we deduce that $b$ satisfies the asymmetric Lusin--Lipschitz estimate with $h\in L^1_\loc([0,+\infty;L^{d,\infty})$. A.e. uniqueness of the Cauchy problem thus follows from Proposition \ref{prop:sufficient_ae_uniqueness}.
	
	Arguing as in Proposition \ref{prop:sufficient_ae_uniqueness}, by Gr\"onwall's lemma we deduce that
	\begin{align*}
		\sup_{t\in [0,T]} |X_t(x)-X_t(y)| \leq \exp\bigg( \int_0^T h_s(X_s(x)) \dd s \bigg) |x-y| =: e^{H_T(x)} |x-y|.
	\end{align*}
	It remains to prove the estimate for $H_T$. Since $d\geq 2$, we can invoke the upcoming Lemma \ref{lem:inverse_interpolation_lorentz} from Section \ref{subsec:lorentz}, for any $\delta>0$, to decompose $h$ as
	\begin{align*}
		h = h^> + h^<, \quad h^>_t(x)= h_t(x) \mathbbm{1}_{|h_t(x)|\geq \delta \| h_t\|_{L^{d,\infty}}},
		\quad h^<_t(x)= h_t(x) \mathbbm{1}_{|h_t(x)|< \delta \| h_t\|_{L^{d,\infty}}}
	\end{align*}
	where
	\begin{align*}
		\int_0^T \| h^>_t\|_{L^\infty} \dd t \leq \delta \int_0^T \| \nabla b_t\|_{L^{d,1}} \dd t, \quad
		\int_0^T \| h^<_t\|_{L^\infty} \dd t \lesssim \delta^{1-d} \int_0^T \| \nabla b_t\|_{L^{d,1}} \dd t.
	\end{align*}
	As the flow $X$ is measure-preserving, $H_T$ admits the same decomposition $H_T=H^<_T+H^>_T$, with bounds
	\begin{align*}
		\|H^<_T\|_{L^\infty} \leq \delta \int_0^T \| \nabla b_t\|_{L^{d,1}} \dd t, \quad 
		\|H^>_T\|_{L^1} \lesssim \delta^{1-d} \int_0^T \| \nabla b_t\|_{L^{d,1}} \dd t.
	\end{align*}
	Set $C= \int_0^T \| \nabla b_t\|_{L^{d,1}} \dd t$; for any $a>0$, employing the above decomposition with $\delta = a/(2C)$, it holds
	\begin{align*}
	\mathscr{L}^d( x: |H_T(x)|>a)
	\leq \mathscr{L}^d\Big( x: |H_T^>(x)|>\frac{a}{2}\Big)
	\leq \frac{2\, \| H^>_T\|_{L^1}}{a} 
	\lesssim \frac{C\, \delta^{1-d}}{a} 
	\lesssim C^d a^{-d}.
	\end{align*}
By the arbitrariness of $a>0$ and the definition of $L^{d,\infty}$, the conclusion follows.
\end{proof}

\begin{remark}
In Proposition \ref{prop:asymmetric_lorentz} and Corollary \ref{cor:asymmetric_lusin_lipschitz_flow} we focused on $\nabla\cdot b=0$, $\nabla b\in L^{d,1}$, since this is the case that will be relevant for us in relation to $3$D Navier--Stokes.
Up to minor modifications, one can consider the case $\nabla\cdot b\in L^1_T L^\infty$ and/or $\nabla b\in L^1_T L^p$ with $p>d$, upon relying on \cite[Lemma 5.1]{CarCri2021} instead. In this case, \eqref{eq:asymmetric_lusin_flow} still holds, this time with
	\begin{align*}
		\|H_T\|_{L^p} \lesssim \exp\bigg( \int_0^T \| \div\, b_s\|_{L^\infty} \dd s\bigg) \int_0^T \| \nabla b_s\|_{L^p} \dd s.
	\end{align*}
	%
	%
\end{remark}

\subsection{Quantitative convergence of Picard iterations}\label{subsec:picard}

Let $b$ to satisfy assumption \eqref{eq:asymmetric_lusin_lipschitz_flow_assumption} from Corollary \ref{cor:asymmetric_lusin_lipschitz_flow}; for simplicity, let us work on a finite interval $[0,T]$ and fix $\gamma\in \cC_T$.

Since $b\in L^1_t C^0$, for any fixed $x\in\R^d$ we can define the map
\begin{align*}
	F(z;x)_t := x + \int_0^t b_s(z_s) \dd s +\gamma_t
\end{align*}
and correspondingly the set of solutions to the Cauchy problem:
\begin{align*}
	S(x;\gamma,b):=\big\{z\in\cC_T: z=F(z;x)\big\}.
\end{align*}
By Peano's theorem, $S(x;\gamma,b)$ is non-empty and compact for every $x\in \R^d$; by Corollary \ref{cor:asymmetric_lusin_lipschitz_flow}, it is a singleton for Lebesgue a.e. $x$.
It is then natural to wonder, also in view of numerical schemes, whether the Picard iterations of $F$ will converge to the unique solution.
We provide here a positive, quantitative answer, by leveraging again on asymmetric estimates; to the best of our knowledge, this kind of problem had not been considered so far in the literature.

\begin{theorem}\label{thm:picard}
	Let $b$ to satisfy assumption \eqref{eq:asymmetric_lusin_lipschitz_flow_assumption}, $\gamma\in \cC_T$ and $X$ denote the associated RLF. For every $x\in\R^d$, define recursively $Z^{(n)}(x)$ by $Z^{(0)}(x)=x+\gamma$, $Z^{(n+1)}(x)=F(Z^{(n)};x)$. Let $H_T$ be the function defined in Corollary \ref{cor:asymmetric_lusin_lipschitz_flow}. Then it holds	
	\begin{equation}\label{eq:estim_picard}
		\sup_{n\in\N}\, e^n \| Z^{(n)}(x) - X(x)\|_{\cC_T} \leq e^{e H_T(x)} \int_0^T \| b_t\|_{C^0} \dd t \quad \text{for a.e. }x\in\R^d.
	\end{equation}
\end{theorem}

\begin{proof}
Recall the function $h$ from Corollary \ref{cor:asymmetric_lusin_lipschitz_flow}; define a distance on $\cC_T$ by
\begin{align*}
	d_T(z,w):=\sup_{t\in [0,T]} \bigg\{ \exp\Big( -e \int_0^t h_s(X^x_s) \dd s \Big) |z_t-w_t| \bigg\}.
\end{align*}
Then for any $w\in C_T$ it holds
\begin{align*}
	|F(w;x)_t - X(x)_t|
	& \leq \int_0^t h_s(X_s(x)) |w_s-X_s(x)| \dd s\\
	& \leq d_T(F(w;x),X(x)) \int_0^t h_s(X_s(x)) \exp\Big( e \int_0^t h_s(X_s(x))\Big) \dd s\\
	& \leq e^{-1} d_T(F(w;x),X(x)) \exp\Big( 2 \int_0^t h_s(X_s(x)) \dd s\Big),
\end{align*}
the computation being valid for any $x\in\R^d$ such that $H_T(x)<\infty$. It follows that
\begin{equation}
	d_T(F(w;x),X(x)) \leq e^{-1} d_T(w,X(x)) \quad \forall\, w\in C_T.
\end{equation}
Iterating, we deduce that $d_T(Z^{(n)}(x), X(x)) \leq e^{-n} d_T(Z^{(0)}(x),X(x))$ and so that
\begin{align*}
	\sup_{t\in [0,T]} |Z^{(n)}_t(x) - X_t(x)|
	\leq e^{e H_T(x) - n} d_T(Z^{(0)},X^x)
	\leq e^{e H_T(x) - n} \sup_{t\in [0,T]} |Z^{(0)}_t(x) - X_t(x)|
\end{align*}
By the definition of $Z^{(0)}(x)$ and the fact that $X$ solves the ODE, we have
\begin{align*}
	\sup_{t\in [0,T]} |Z^{(0)}_t(x) - X_t(x)|
	\leq \int_0^T |b_t(X_t(x))| \dd t
	\leq \int_0^T \| b_t \|_{C^0} \dd t.
\end{align*}
Combined with the above estimate, this yields \eqref{eq:estim_picard}.
\end{proof}

\begin{remark}\label{rem:picard}
By \eqref{eq:estim_picard}, for a.e. $x\in\R^d$, the Picard iterations $Z^{(n)}(x)$ converge exponentially fast to $x$. On the other hand, since we only know that $H_T\in L^{d,\infty}$, the r.h.s. of \eqref{eq:estim_picard} can be very large. We can estimate the size of the ``bad set'' where $Z^{(n)}(x)$ is not close to $X(x)$ by 
\begin{align*}
	\mathscr{L}^d\Big(x: \sup_{t\in [0,T]} & |Z^{(n)}_t(x) - X_t(x)| > \lambda\Big)
	\leq \mathscr{L}^d\Big(x: e^{-n +e H_T(x)} \| b\|_{L^1_t C^0} > \lambda\Big)\\
	&= \mathscr{L}^d\Big(x: H_T(x) > e^{-1} (n+\log \lambda - \log \| b\|_{L^1_t C^0})\Big)\\
	& \lesssim \| H_T\|_{L^{d,\infty}}^d (n+\log \lambda - \log \| b\|_{L^1_t C^0})^{-d}\\
	& \lesssim \Big( \int_0^T \|\nabla b_s\|_{L^{d,1}} \dd s \Big)^d \Big(n+\log \lambda - \log \int_0^T \| b_s\|_{C^0} \dd s\Big)^{-d}.
\end{align*}
For instance, taking $\lambda=e^{-n/2}$, we see that asymptotically in $n$ it holds
\begin{align*}
	\mathscr{L}^d\Big(x: \sup_{t\in [0,T]} & |Z^{(n)}_t(x) - X_t(x)| > e^{-n/2} \Big) \lesssim \Big( \int_0^T \|\nabla b_s\|_{L^{d,1}} \dd s \Big)^d n^{-d}. 
\end{align*}
\end{remark}

\section{Interpolation-type inequalities in Lorentz and Besov spaces}\label{sec:interpolation}

This section is devoted to a class of functional inequalities, in order to verify condition $\nabla b\in L^1_\loc([0,+\infty);L^{d,1})$ for PDEs of interest.
In the specific case of $3$D Navier--Stokes, all that will be needed is Corollary \ref{cor:refined_inequality} below, but the results contained here are of independent interest.

The driving principle behind the results contained here is quite general: given some Banach spaces of functions $E_0$, $E_1$ and $\theta\in (0,1)$, the resulting class $\cE_\theta$ of function spaces satisfying a $\theta$-inequality, namely such that
\begin{align*}
	\| f\|_{\cE^\theta} \leq \| f\|_{E_0}^{1-\theta}\, \| f\|_{E_1}^\theta\quad \forall\, f\in E_0\cap E_1
\end{align*}
contains better spaces than those one would naively expect by classical interpolation results.
Morally, one should expect $\cE^\theta$ to contain the ``best possible spaces'' which are consistent with the correct scaling behaviour (dictated by $\theta$) and possibly Sobolev-type embeddings. This is indeed the case for:
\begin{itemize}
\item[i)] Agmon's inequality, cf. Lemma \ref{lem:agmon} below, where $E_0=\dot H^{s_0}$, $E_1=\dot H^{s_1}$ and $C^0\in \cE_\theta$ for the critical values of $\theta$ such that the embedding $\dot H^{d/2}\hookrightarrow L^\infty$ fails.
\item[ii)] Lemma \ref{lem:interpolation_lorentz} for Lorentz spaces, where $E_0=L^{p_0,\infty}$, $E_1=L^{p_1,\infty}$ and $\cE_\theta$ contains $L^{p_\theta,1}$, rather than just the naive guess $L^{p_\theta,\infty}$. 
\item[iii)] The known inequality \eqref{eq:interpolation-besov} in Besov spaces, where $E_0= B^0_{p,\infty}$, $E_1= B^1_{p,\infty}$ and $B^\theta_{p,1}\in \cE_\theta$, rather than just $B^\theta_{p,\infty}$.
\item[iv)] Theorem \ref{thm:refined_besov}, where $E_0=B^0_{p,\infty}$, $E_1=B^1_{p,\infty}$ and $L^{p_\theta,1}\in\cE_\theta$, rather then just $B^0_{p_\theta,\infty}$ as one would obtain from elementary interpolation and Besov embeddings.
\end{itemize}

We claim no novelty about the results contained in this section, which are likely known to experts; still, we couldn't find explicit references, which is why we provide self-contained proofs here for completeness.

We start with a proof of Agmon's inequality, which will be relevant for $3$D Navier--Stokes as already noticed in \cite{FGT1981}.
The natural way to state and prove it is by actually considering the Fourier--Lebesgue space $\cF L^1:=\{f: \hat f\in L^1\}$.

\begin{lemma}\label{lem:agmon}
	Let $f\in \dot{H}^{s_0}(\R^d)\cap \dot{H}^{s_1}(\R^d)$, for parameters $0\leq s_0<d/2<s_1$; let $\theta\in (0,1)$ be such that
	\begin{align*}
		\frac{d}{2} = \theta s_1 + (1-\theta) s_0.
	\end{align*}
	Then $f\in C^0\cap \cF L^1$ and there exists $C=C(s_0,s_1,d)>0$ such that
	\begin{equation}\label{eq:agmon_generalised}
		\| f\|_{C^0}\lesssim \| f\|_{\cF L^1} \leq C \| f\|_{\dot{H}^{s_0}}^{1-\theta} \| f\|_{\dot H^{s_1}}^{\theta}.
	\end{equation}
\end{lemma}

\begin{proof}
	The embedding $\cF L^1\hookrightarrow C^0$ is a basic consequence of dominated convergence, once $f$ is expressed as the antitransform of $\hat f$; so we only need to bound $\| f\|_{\cF L^1}$.
	We can assume $f$ smooth, as the general case follows by density. For any $M>0$, it holds
	\begin{align*}
		\int_{\R^d} |\hat{f}(\xi)| \dd \xi
		& \leq \int_{|\xi|\leq M} |\xi|^{-s_0} |\xi|^{s_0} |\hat{f}(\xi)| \dd \xi + \int_{|\xi|> M} |\xi|^{-s_1} |\xi|^{s_1} |\hat{f}(\xi)| \dd \xi\\
		& \leq \Big( \int_{|\xi|\leq M} |\xi|^{-2s_0} \dd \xi\Big)^{1/2} \| f\|_{\dot{H}^{s_0}} + \Big( \int_{|\xi|> M} |\xi|^{-2s_1} \dd \xi \Big)^{1/2} \| f\|_{\dot{H}^{s_1}}\\
		& \lesssim M^{d/2-s_0} \| f\|_{\dot{H}^{s_0}} + M^{d/2-s_1} \| f\|_{\dot{H}^{s_1}}
	\end{align*}
	Estimate \eqref{eq:agmon_generalised} follows by optimizing in $M$, namely taking 
	\begin{align*}
		M = \Big(\frac{\| f\|_{\dot{H}^{s_1}}}{\| f\|_{\dot{H}^{s_0}}} \Big)^{\frac{1}{s_1-s_0}}
		= \Big(\frac{\| f\|_{\dot{H}^{s_1}}}{\| f\|_{\dot{H}^{s_0}}} \Big)^{\frac{\theta}{d/2-s_0}}. \qquad \qquad \qedhere
	\end{align*}
\end{proof}

\subsection{Inequalities in Lorentz spaces}\label{subsec:lorentz}

We recall that for $p\in (0,\infty)$, $q\in (0,\infty]$, the Lorentz space $L^{p,q}$ is defined by
\begin{align*}
	\| f\|_{p,q}^q:= \int_0^{+\infty} r^{q-1} \mathscr{L}^d(\{x:|f(x)|>r\})^{\frac{q}{p}} \dd r
\end{align*}
with usual convention for $q=\infty$.

\begin{lemma}\label{lem:interpolation_lorentz}
	Let $p_0,p_1\in [1,\infty]$ with $p_0<p_1$ and $\theta\in (0,1)$; set $\frac{1}{p_\theta}=\frac{1-\theta}{p_0}+\frac{\theta}{p_1}$. Then for any $f\in L^{p_1,\infty}\cap L^{p_0,\infty}$ it holds $f\in L^{p_\theta,1}$ with
	\begin{align*}
		\| f\|_{L^{p_\theta,1}} \leq C \| f\|_{L^{p_0,\infty}}^{1-\theta} \| f\|_{L^{p_1,\infty}}^{\theta}
	\end{align*}
	for a constant $C=C(p_0,p_1,\theta)$ (see the end of the proof for its explicit formula).
\end{lemma}

\begin{proof}
	By homogeneity, we may assume $\| f\|_{L^{p_1,\infty}}=1$. For a parameter $R>0$ to be fixed later, by the definition of Lorentz norms, it holds
	\begin{align*}
		\| f\|_{L^{p_\theta,1}}
		& = \int_0^{R}  \mathscr{L}^d(\{x:|f(x)|>r\})^{\frac{1}{p_\theta}} \dd r + \int_R^{+\infty}  \mathscr{L}^d(\{x:|f(x)|>r\})^{\frac{1}{p_\theta}} \dd r\\
		& \leq \| f\|_{L^{p_0,\infty}}^{\frac{p_0}{p_\theta}} \int_0^R r^{-\frac{p_0}{p_\theta}} \dd r + \int_R^{+\infty} r^{-\frac{p_1}{p_\theta}} \dd r\\
		& = \frac{p_\theta}{p_\theta-p_0}\, \| f\|_{L^{p_0,\infty}}^{\frac{p_0}{p_\theta}}\, R^{1-\frac{p_0}{p_\theta}} + \frac{p_\theta}{p_1-p_\theta}\, R^{1-\frac{p_1}{p_\theta}}.
	\end{align*}
	Taking $R$ such that the two terms above are equal, namely such that
	\begin{align*}
		R^{-\frac{1}{p_\theta}} = \Big( \frac{p_1-p_\theta}{p_\theta-p_0}\Big)^{\frac{1}{p_1-p_0}} \| f\|_{L^{p_0,\infty}}^{\frac{p_0}{p_1-p_0}\,\frac{1}{p_\theta}},
	\end{align*}
	we then find
	\begin{align*}
		\| f\|_{L^{p_\theta,1}}
		\leq 2 \frac{p_\theta}{p_1-p_\theta} \Big( \frac{p_1-p_\theta}{p_\theta-p_0}\Big)^{\frac{p_1-p_\theta}{p_1-p_0}} \| f\|_{L^{p_0,\infty}}^{\frac{p_0}{p_1-p_0}\,\frac{p_1-p_\theta}{p_\theta} }.
	\end{align*}
	Observing that by the definition of $p_\theta$ it holds that
	\begin{align*}
		\frac{p_1-p_\theta}{p_\theta-p_0} = \frac{p_1}{p_0} \frac{1-\theta}{\theta}, \quad
		\frac{p_1-p_\theta}{p_\theta} = (1-\theta) \frac{p_1-p_0}{p_0}, \quad
		\frac{p_0}{p_1-p_0} \frac{p_1-p_\theta}{p_\theta} = 1-\theta,
	\end{align*}
	we reach the desired conclusion with explicit constant
	\begin{align*}
	 	C(p_0,p_1,\theta):=
		2 \frac{p_0}{p_1-p_0} \frac{1}{1-\theta} \Big( \frac{p_1}{p_0} \frac{1-\theta}{\theta}\Big)^{\frac{p_1}{p_0} \frac{1-\theta}{\theta}}.\qquad \qquad \qedhere
	\end{align*}
\end{proof}

The most important consequence of Lemma \ref{lem:interpolation_lorentz} for us, in view of Section \ref{sec:main-proof}, is the following.

\begin{corollary}\label{cor:refined_inequality}
	For any $f\in H^1(\R^3)$, it holds $\| f\|_{L^{3,1}} \lesssim \| f\|_{L^2}^{1/2} \| \nabla f\|_{L^2}^{1/2}$.
\end{corollary}

\begin{proof}
By Lorentz and Sobolev embeedings in $\R^3$, it holds
\begin{align*}
	\| f\|_{L^{2,\infty}} \leq \| f\|_{L^2}, \quad \| f\|_{L^{6,\infty}}\leq \| f\|_{L^6} \lesssim \| \nabla f\|_{L^2}.
\end{align*}
The conclusion follows by applying Lemma \ref{lem:interpolation_lorentz} with $\theta=1/2$, $p_0=2$, $p_1=6$.
\end{proof}

The next lemma played a relevant role in Corollary \ref{cor:asymmetric_lusin_lipschitz_flow}

\begin{lemma}\label{lem:inverse_interpolation_lorentz}
	Let $f\in L^{p,\infty}$ for some $p\in (0,\infty)$. Then for any $\delta\in (0,\infty)$ and any $q<p$, there exists a decomposition $f=f^<+f^>$ such that $f^<\in L^\infty$, $f^<\in L^q$ and 
	\begin{align*}
		\| f^<\|_{L^\infty} \leq \delta\, \| f\|_{L^{p,\infty}}, \quad
		\| f^>\|_{L^q} \lesssim_{p,q} \delta^{1-\frac{p}{q}}\, \| f\|_{L^{p,\infty}}.
	\end{align*}	 
\end{lemma}

\begin{proof}
	Set $f^< := f \mathbbm{1}_{|f|\leq \delta \| f\|_{L^{p,\infty}}}$, $f^> := f \mathbbm{1}_{|f|> \delta \| f\|_{L^{p,\infty}}}$; clearly $f=f^<+f^>$ and $f^<$ satisfies the required bound. Moreover
	\begin{align*}
		\| f^>\|_{L^q}
		& \sim_q \int_0^\infty \mathscr{L}^d(x: |f^>(x)|>a)\, a^{q-1} \dd a
		= \int_{\delta \| f\|_{L^{p,\infty}}}^\infty \mathscr{L}^d(x: |f(x)|>a)\, a^{q-1} \dd a\\
		& \leq \| f\|_{L^{p,\infty}}^p \int_{\delta \| f\|_{L^{p,\infty}}}^\infty a^{q-p-1} \dd a
		\lesssim_{p,q} \| f\|_{L^{p,\infty}}^q\, \delta^{q-p}
	\end{align*}
	which yields the conclusion.
\end{proof}

\subsection{Inequalities in Besov spaces}\label{subsec:besov}

We work here inhomogeneous Besov spaces $B^s_{p,q}$, defined via Littlewood--Paley blocks $\{\Delta_j\}_{j\geq -1}$, and their homogenous counterparts $\dot B^s_{p,q}$ with blocks $\{\dot\Delta_j\}_{j\in\mathbb{Z}}$, as defined in the monograph \cite{BaChDa2011}.
Notice that the two classes can be related by the embedding $B^s_{p,q}\hookrightarrow \dot B^s_{p,q}$ whenever $s>0$, the converse one holding for $s<0$ instead.

As mentioned, $\theta$-inequalities in Besov spaces can be quite powerful: by \cite[Theorem 2.80]{BaChDa2011}, for any $p\in [1,\infty]$ and $\theta \in (0,1)$ it holds
\begin{equation}\label{eq:interpolation-besov}
	\| f\|_{B^\theta_{p,1}} \lesssim \| f\|^\theta_{B^1_{p,\infty}} \|f\|^{1-\theta}_{B^0_{p,\infty}}.
\end{equation}
Refined inequalities relating Besov spaces and Lorentz spaces have been investigated before in \cite{BahCoh2011} and \cite{ChLeRi2013}; building on these works, we can prove the following.

\begin{theorem}\label{thm:refined_besov}
	Let $p\in [1,\infty)$, $f\in B^1_{p,\infty}$ and $\theta\in (0,1)$ such that $\theta p < d$. Set
	\begin{equation}\label{eq:p_theta}
		\frac{1}{p_\theta} := \frac{1}{p}-\frac{\theta}{d}.
	\end{equation}
	Then there exists a constant $C=C(d,p,\theta)$ such that
	\begin{equation}\label{eq:refined_besov}
		\| f\|_{L^{p_\theta,1}} \leq C \| f\|_{B^1_{p,\infty}}^\theta \| f\|_{B^0_{p,\infty}}^{1-\theta}.
	\end{equation}
\end{theorem}

\begin{remark}
	For $f\in B^0_{p,\infty}\cap B^1_{p,\infty}$, inequality \eqref{eq:interpolation-besov} and Besov embeddings ((\cite[Proposition 2.20]{BaChDa2011}) already imply that $f\in B^\theta_{p,1}\hookrightarrow B^{0}_{p_\theta,1}$; by its definition and Minkowski's inequality, $B^{0}_{p_\theta,1}\hookrightarrow L^{p_\theta}$.
	Theorem \ref{thm:refined_besov} makes the extra step of reaching the sharper $L^{p_\theta,1}$ instead.
\end{remark}

To prove Theorem \ref{thm:refined_besov}, we need to recall the following result.

\begin{theorem}[Theorem 2 from \cite{ChLeRi2013}]\label{thm:chamorro}
Let $\alpha$, $\beta>0$, $p_0,\,p_1\in [1,+\infty]$ with $p_0\neq p_1$. Let
\begin{equation*}
\theta:=\frac{\alpha}{\alpha+\beta}\in (0,1), \quad \frac{1}{p_\theta}:= \frac{1-\theta}{p_0} + \frac{\theta}{p_1}, \quad r\in [1,\infty].
\end{equation*}
If $f\in \dot B^\alpha_{p_0,r}\cap \dot B^{-\beta}_{p_1,r}$ then $f\in L^{p_\theta,r}$ and we have
\begin{equation}
	\| f\|_{L^{p_\theta,r}} \lesssim \| f\|_{\dot{B}^\alpha_{p_0,r}}^{1-\theta} \| f\|_{\dot{B}^{-\beta}_{p_1,r}}^\theta
\end{equation}
\end{theorem}

\begin{proof}[Proof of Theorem \ref{thm:refined_besov}]
Consider some parameters $\gamma\in (0,1)$, $\delta>1-\gamma$ to be fixed later, such that $\delta p \leq d$. Since $f\in B^1_{p,\infty}$, by Besov embeddings (\cite[Proposition 2.20]{BaChDa2011}) it holds
\begin{align*}
	f\in B^\gamma_{p,1}\hookrightarrow \dot B^\gamma_{p,1}, \quad
	f\in B^{1-\gamma}_{p,1}\hookrightarrow \dot B^{1-\gamma}_{p,1}\hookrightarrow \dot B^{1-\gamma-\delta}_{\tilde p,1} \quad \text{for}\quad \frac{1}{\tilde p}=\frac{1}{p}-\frac{\delta}{d}.
\end{align*}
Since $\tilde p\neq p$, by setting
\begin{equation*}
	\tilde \theta := \frac{\gamma}{2\gamma+\delta-1},\quad
	\frac{1}{p_\ast} := \frac{1-\tilde\theta}{p}+\frac{\tilde\theta}{\tilde p}=\frac{1}{p}-\frac{\tilde\theta \delta}{d}
\end{equation*}
we can apply consecutively Theorem \ref{thm:chamorro} (with $r=1$, $\alpha=\gamma$, $\beta=1-\gamma-\delta$) and \eqref{eq:interpolation-besov} to find
\begin{equation*}
	\| f\|_{L^{p_\ast,1}}
	\lesssim \| f\|_{B^\gamma_{p,1}}^{1-\tilde\theta} \| f\|_{B^{1-\gamma}_{p,1}}^{\tilde\theta}
	\lesssim \| f\|_{B^\gamma_{1,\infty}}^{\gamma(1-\tilde\theta) + (1-\gamma)\tilde \theta} \| f\|_{B^0_{p,\infty}}^{(1-\gamma)\tilde\theta + \gamma(1-\tilde\theta)}.
\end{equation*}
In order to deduce \eqref{eq:refined_besov}, it remains to show that we can choose $\delta$, $\gamma$ so that
\begin{align}
	& \frac{1}{p_\ast}=\frac{1}{p_\theta}\quad \Leftrightarrow \quad \delta \tilde\theta = \theta, \label{eq:proof-interpol-1}\\
	& \gamma(1-\tilde\theta) + (1-\gamma)\tilde \theta = \theta, \label{eq:proof-interpol-2}\\
	& \gamma\in (0,1), \quad \delta>1-\gamma,\quad \delta p\leq d. \label{eq:proof-interpol-3}
\end{align}
Let us set $\delta=1-\gamma+\eps$ for a new parameter $\eps>0$; algebraic manipulations show that conditions \eqref{eq:proof-interpol-1}-\eqref{eq:proof-interpol-2} are actually equivalent and correspond to
\begin{equation}\label{eq:proof-interpol-4}
	\theta = \frac{\gamma(1-\gamma+\eps)}{\gamma+\eps} =: F(\gamma,\eps)
\end{equation}
$F$ in continuous on $(0,1)\times [0,1]$ and $F(\gamma,0)=1-\gamma$; by continuity, we can find $\tilde\eps_1,\,\tilde\eps_2>0$ arbitrarily small such that
\begin{align*}
	F(1-\theta-2\tilde\eps_1,\tilde\eps_2)> \theta + \eps_1, \quad F(1-\theta+2\tilde\eps_1,\tilde\eps_2)< \theta - \eps_1.
\end{align*}
The intermediate value theorem guarantees the existence of $\bar\gamma\in (1-\theta-2\tilde\eps_1,1-\theta+2\tilde\eps_1)$ and $\bar{\eps}=\tilde\eps_2>0$ such that \eqref{eq:proof-interpol-4} holds. Since by assumption $\theta<d/p$ and the parameters $\tilde\eps_i$ can be arbitrarily small, the coefficients can be chosen so that \eqref{eq:proof-interpol-3} is satisfied as well.
\end{proof}

\begin{remark}\label{rem:refined_lebesgue_proof}
Theorem \ref{thm:refined_besov} provides an alternative derivation of Corollary \ref{cor:refined_inequality}: choosing  $p=2$, $d=3$, $\theta=1/2$ and applying the embeddings
\begin{align*}
	H^1=B^1_{2,2}\hookrightarrow B^1_{2,\infty},\quad L^2=B^0_{2,2}\hookrightarrow B^0_{2,\infty}
\end{align*}
one can deduce again that $\| f\|_{L^{3,1}} \lesssim \| f\|_{L^2}^{1/2} \| f\|_{H^1}^{1/2}$.
The proof of Theorem \ref{thm:refined_besov} is a bit convoluted, mostly due to the contraint $p_0\neq p_1$ in Theorem \ref{thm:chamorro}; but in practical cases, one can often find explicit values for $\gamma$ and $\delta$.
In the above example with $p=2$, $d=3$, $\theta=1/2$, one can take $\delta=1$ and any $\gamma\in (0,1)$.
\end{remark}

In light of the discussion from the beginning of Section \ref{sec:interpolation}, we expect a more general version of Theorem \ref{thm:refined_besov} to be true: for any $s_1>s_0\geq 0$, $p\in [1,\infty)$ and any $\theta\in (0,1)$, setting
\begin{align*}
	s_\theta:=\theta s_1 + (1-\theta)s_0, \quad \frac{1}{p_\theta}:=\frac{1}{p}-\frac{s_\theta}{d}
\end{align*}
(under the assumption $s_\theta p \leq d$), it should hold
\begin{equation}
	\| f\|_{L^{p_\theta,1}} \lesssim \| f\|_{B^{s_1}_{p,\infty}}^\theta  \| f\|_{B^{s_0}_{p,\infty}}^{1-\theta}.
\end{equation}
We leave this question for future research, since Corollary \ref{cor:refined_inequality} is enough for our purposes.

\section{Applications to $3$D Navier--Stokes}\label{sec:main-proof}

We are now ready to apply the abstract results from Sections \ref{sec:flows}-\ref{sec:interpolation} to the specific case where the drift $b$ is given by a solution to the Navier--Stokes equations on $\R^3$. Recall that in this paper we focus exclusively on Leray solutions, in the sense of Definition \ref{defn:leray_solution}.

\subsection{Regularity estimates for Leray solutions}\label{subsec:regularity_leray}

We start by recalling the ``higher order a priori estimates'' for Leray solutions first obtained in \cite{FGT1981}. Curiously, all versions we could find in the literature concern either the torus or bounded domains (cf. \cite{FGT1985}, \cite[Lem. 8.15]{RRS2016} and \cite[Thm. 4.2]{Temam1995}). For this reason, we give a short self-contained proof in full space.

\begin{lemma}\label{lem:FGT}
Let $u_0\in L^2$, $u$ be a Leray solution to the Navier--Stokes equations on $\R^3$. Then for all $T\geq 1$, it holds
\begin{equation}\label{eq:FGT}
	\int_0^T \| u_t\|_{\dot H^2}^{2/3} \dd t \lesssim (1+\|u_0\|_{L^2}^2) T^{1/3}.
\end{equation}
\end{lemma}

\begin{proof}
	For simplicity, we manipulate $u$ as if it were a smooth solution; rigorously, one should first perform all estimates at the level of the mollified systems \eqref{eq:leray_scheme} and then pass to the limit as $\eps\to 0$ using lower semicontinuity of norms, but let us omit this standard passage.
	
	By standard computations (see e.g. \cite[Thm. 3.4, eq. (3.41)]{BedVic2022}) we know that
	\begin{align*}
		\frac{\dd }{\dd t} \| \nabla u\|_{L^2}^2 + \|\Delta u\|_{L^2}^2 \leq C \| \nabla u\|_{L^2}^6
	\end{align*}
	Let $\delta>0$ to be fixed later; dividing both sides by $(\delta + \|\nabla u\|_{L^2}^2)^2$, we obtain
	\begin{align*}
	\frac{\dd}{\dd t} \Big( - \frac{1}{\delta + \| \nabla u\|_{L^2}^2}\Big)
	= \frac{1}{(\delta + \| \nabla u\|_{L^2}^2)^2} \frac{\dd }{\dd t} \| \nabla u\|_{L^2}^2
	\leq - \frac{\|\Delta u\|_{L^2}^2}{(\delta + \| \nabla u\|_{L^2}^2)^2} + C \| \nabla u \|_{L^2}^2
	\end{align*}
	Integrating in time and using the energy inequality \eqref{eq:leray_strong_energy}, after some rearrangements we find
	 \begin{equation}\label{eq:FGT_intermediate}
		\int_0^T \frac{\|\Delta u_t\|_{L^2}^2}{(\delta + \| \nabla u_t\|_{L^2}^2)^2} \dd t
		\leq C \| u_0\|_{L^2}^2 + \frac{1}{\delta + \| \nabla u_T\|_{L^2}^2}
		\leq	 C \| u_0\|_{L^2}^2 + \delta^{-1}.
	 \end{equation}
	 By H\"older's inequality we then have
	 \begin{align*}
	 	\int_0^T \|\Delta u_t\|_{L^2}^{2/3} \dd t
	 	& \leq \bigg( \int_0^T \frac{\|\Delta u_t\|_{L^2}^2}{(\delta + \|\nabla u_t\|_{L^2}^2)^2} \dd t \bigg)^{1/3} \bigg( \int_0^T (\delta + \|\nabla u_t\|_{L^2}^2) \dd t \bigg)^{2/3}\\
	 	& \leq  \big(C \| u_0\|_{L^2}^2 + \delta^{-1}\big)^{1/3} \big( \delta T + \|u_0\|_{L^2}^2 \big)^{2/3}
	 \end{align*}
	 where we used \eqref{eq:FGT_intermediate} and the energy inequality. Choosing $\delta=T^{-1}$, the conclusion follows.
\end{proof}

\begin{remark}
	The usual proof (cf. \cite[Lem. 8.15]{RRS2016}) would perform the computation for $\delta=0$, resulting in the estimate
	\begin{equation}\label{eq:FGT_alternative}
	\int_0^T \|\Delta u_t\|_{L^2}^{2/3} \dd t
	 	\leq  \bigg(C \| u_0\|_{L^2}^2 + \frac{1}{\|\nabla u(T)\|_{L^2}^2}\,\bigg)^{1/3} \|u_0\|_{L^2}^{4/3}
	\end{equation}
whenever $\|\nabla u(T)\|_{L^2}>0$. Recalling that Leray solutions eventually become small in $L^3$ (e.g. by the arguments in \cite[Lem. 6.13 and Thm. 8.1]{RRS2016}), Kato's results \cite[Thm. 4]{Kato1984} imply an asymptotic decay at least of the form $\| \nabla u(T)\|_{L^2} \lesssim T^{-1/2}$ as $T\to\infty$. Therefore as $T\to\infty$, the bound \eqref{eq:FGT} obtained by optimizing in $\delta$ is potentially sharper than \eqref{eq:FGT_alternative}. We do not know whether it can be further improved.
\end{remark}

\begin{corollary}\label{cor:onesided_navier_stokes}
Let $u_0\in L^2$ and $u$ be an associated Leray solution. Then for any $T>0$ it holds
\begin{align*}
	\int_0^T \big(\| \nabla u_t\|_{L^{3,1}} + \| u_t\|_{\cF L^1}\big) \dd t \lesssim \| u_0\|_{L^2}^{\frac12} (1+\|u_0\|_{L^2}^2)^{3/4} T^{1/4}.
\end{align*}
Moreover there exists $h\in L^1_\loc([0,+\infty);L^{3,\infty})$ such that
\begin{equation}\label{eq:onesided_navier_stokes}
	|u_t(x)-u_t(y)|\leq h_t(x) |x-y| \quad \text{for a.e. } t\in [0,+\infty) \text{ and all } x,y\in\R^3.
\end{equation}
\end{corollary}

\begin{proof}
	By Corollary \ref{cor:refined_inequality}, for a.e. $t$ it holds  $\| \nabla u_t\|_{L^{3,1}} \lesssim \| \nabla u_t\|_{L^2}^{1/2} \| \nabla^2 u_t\|_{L^2}^{1/2}$.
	By Lemma \ref{lem:FGT} and H\"older's inequality, we then have
	\begin{align*}
		\int_0^T \| \nabla u_t\|_{L^{3,1}} \dd t
		& \lesssim \int_0^T \| \nabla u_t\|_{L^2}^{1/2} \| u_t\|_{\dot H^2}^{1/2} \dd t\\
		& \lesssim \Big(\int_0^T \| \nabla u_t\|_{L^2}^2 \dd t\Big)^{\frac14} \Big(\int_0^T \| u_t\|_{\dot H^2}^{\frac23} \dd t\Big)^{\frac34}\\
		& \lesssim \| u_0\|_{L^2}^{\frac12} (1+\|u_0\|_{L^2}^2)^{3/4} T^{1/4}.
	\end{align*}
	By applying Proposition \ref{prop:asymmetric_lorentz}, we deduce the existence of $h\in L^1_\loc([0,+\infty);L^{3,\infty})$ as desired.
	The bound for $\| u_t\|_{\cF L^1}$ is similar and follows from Lemma \ref{lem:agmon}.
\end{proof}

The previous results concern local integrability in time; under additional assumptions on $u_0$, they can be upgraded to global-in-time results.

\begin{corollary}\label{cor:shonbeck}
	Let $u$ be a Leray solution associated to $u_0\in L^1\cap L^2$. Then
\begin{align*}
	\int_0^{+\infty} \big(\| \nabla u_t\|_{L^{3,1}} + \| u\|_{\cF L^1}\big) \dd t <\infty.
\end{align*}	
\end{corollary}

\begin{proof}
Since $u_0\in L^1$, by the energy inequality \eqref{eq:leray_strong_energy} and \cite[Thm. 3.1]{Schonbeck1986} we know that
\begin{equation}\label{eq:schonbek}
	\| u_t\|_{L^2} \leq \| u_0\|_{L^2}, \quad \| u_t\|_{L^2} \lesssim t^{-3/4}.
\end{equation}
Let $\Gamma\subset [0,+\infty)$ denote the full Lebesgue measure set such that $u_{s+\cdot}$ is still a Leray solution for any $s\in\Gamma$.
By Corollary \ref{cor:onesided_navier_stokes} and \eqref{eq:schonbek}, for any $s\in\Gamma$ and $T>s$ it holds
\begin{align*}
	\int_s^T \| \nabla u_t\|_{L^{3,1}} \dd t
	\lesssim \| u_s\|_{L^2}^{\frac12} (1+\|u_s\|_{L^2}^2)^{3/4} T^{1/4}
	\lesssim s^{-3/8} T^{1/4}
\end{align*}
Now take a sequence $s_n\subset \Gamma$ such that $|s_n-2^n|<1$, $T_n=s_{n+1}$, then
\begin{align*}
	\int_2^\infty \| \nabla u_t\|_{L^{3,1}} \dd t
	= \sum_{n=0}^\infty
	 \int_{s_n}^{s_{n+1}} \| \nabla u_t\|_{L^{3,1}} \dd t
	  \lesssim \sum_{n=0}^\infty 2^{-n(3/8-1/4)} <\infty.\qquad \qquad \qedhere
\end{align*}
\end{proof}

%
%

\subsection{Proof of the main results}\label{subsec:proofs_main}

We are now ready to present the proofs of our main results. To formulate them precisely, we need some additional notations.

Let $b\in L^1_T C^0$ and $\gamma\in \cC_T$. Similarly to Section \ref{subsec:picard}, for any $s\in [0,T)$ let us define
\begin{equation}\label{eq:set_S_x,b,gamma}
	S(x;s,b,\gamma):= \Big\{ z\in C([s,T];\R^d):\ z_t = x + \int_s^t b_r(z_r) \dd r+\gamma_t -\gamma_s\quad \forall\, r\in [s,T] \Big\}
\end{equation}
namely the set all possible solutions to the ODE on $[s,T]$, which is non-empty and compact by Peano's theorem.
For $z\in\cC_T$ and $b\in\mathscr{W}^p_T$, we introduce the translation operators $(\tau_s z)_t:= z_{t+s}$ (similarly $\tau_s b$) and $(\theta_s z)_t=z_{t+s}-z_t$.\footnote{Technically if $z\in \cC_T$, then $\tau_s z\in \cC_{T-s}$, similarly for $b\in\mathscr{W}^p_T$; but let us ignore this technicality by extending $\gamma$, $b$ trivially after $T$.} It follows that
\begin{align*}
	z\in S(x;0,b,\gamma)\quad \Rightarrow\quad
	z\in S(z_s;s,b,\gamma\cdot-\gamma_s)\quad \Rightarrow\quad
	\tau_s z\in S(z_s;0,\tau_s b,\theta_s\gamma)
\end{align*}
for all $s\in [0,T]$.

Inspired by the concept of \emph{path-by-path uniqueness} first introduced by Davie \cite{Davie2007} (see \cite{Flandoli2009} for a deeper discussion), we will consider the following solution concepts. 

\begin{definition}\label{defn:pbp_uniqueness}
	Let $b$ as above, $\mu\in \cP(\cC_T)$.
	\begin{itemize}
	\item[i)] Let $x\in\R^d$ be fixed. We say that \emph{path-by-path uniqueness} holds for $(x,b,\mu)$ if
	\begin{align*}
		\mu\big( \big\{\gamma\in \cP(\cC_T): \text{Card}(S(x;0,b,\gamma))=1 \big\} \big)=1.
	\end{align*}
	\item[ii)] We say that \emph{path-by-path a.e. uniqueness} holds for $(b,\mu)$ if
	\begin{align*}
		\mu\big( \big\{\gamma\in \cP(\cC_T): \text{Card}(S(x;0,b,\gamma))=1 \text{ for Lebesgue a.e. }x\in\R^d\big\} \big)=1.
	\end{align*}
	\end{itemize}
\end{definition} 

Part i) of the definition is in line with \cite{Davie2007}; it is stronger than pathwise uniqueness, as it does not require solutions to be adapted, rather it treats the SDE as an \emph{ODE with random coefficients} by looking at fixed realizations $\gamma$.
Part ii) is a contribution of the present work; the difference between i) and ii) is a bit technical (notice the different order of quantifiers) but relevant. Path-by-path a.e. uniqueness in part ii) is stronger than pathwise uniqueness among random RLFs, as well as pathwise uniqueness among a.e. stochastic flows (AESFs) in the sense of \cite{Zhang2010}; see \cite[Section 5]{Zhao2019} for examples of drifts for which uniqueness of AESFs holds, but weak uniqueness at fixed $x$ fails.

To draw a link between properties i) and ii) in Definition \ref{defn:pbp_uniqueness}, we need some (technical) preparations. Given $\mu\in\cP(\cC_T)$, we denote by $\cF^\mu$ the completion of the Borel sigma algebra of $\cC_T$ w.r.t. $\mu$; we denote by $(\R^d\times \cC_T,\cA,\bar\mu)$ the completion of $(\R^d\times \cC_T, \cF^d\otimes \cF^\mu, \mathscr{L}^d\times \mu)$, where $\cF^d$ denotes the sigma algebra of Lebesgue measurable sets on $\R^d$. Working with such completions will be necessary to invoke Fubini's theorem below, which essentially allows to swap the order in the ``a.e.'' quantifier.

\begin{lemma}\label{lem:relations_uniqueness}
	Let $b\in L^1_T C^0$, $S(x;0,b,\gamma)$ as defined in \eqref{eq:set_S_x,b,gamma}, $\mu\in\cP(\cC_T)$; let $(\R^d\times \cC_T,\cA,\bar\mu)$ be as defined above. Suppose that the set $A:=\{(x,\gamma)\in \R^d\times \cC_T:\, \text{Card}(S(x;0,b,\gamma))=1\}$ is $\cA$-measurable; then the following are equivalent:
	\begin{itemize}
	\item[a)] For Lebesgue a.e. $x$, \emph{path-by-path uniqueness} holds for $(x,b,\mu)$.
	\item[b)] Path-by-path a.e. uniqueness holds for $(b,\mu)$.
	\end{itemize}
\end{lemma}

\begin{proof}
	Since $A$ is $\cA$-measurable, so is its complement $A^c$; Fubini's theorem for complete measures implies that
	\begin{align*}
		\int_{\R^d} \mu\big( & \big\{ \gamma\in \cP(\cC_T): \text{Card}(S(x;0,b,\gamma))\neq1 \big\} \big) \dd x
		= \int_{\R^d} \int_{\cC_T} \mathbbm{1}_{A^c}(x,\gamma) \mu(\d \gamma) \dd x\\
		& = \bar\mu(A^c)
		= \int_{\cC_T} \int_{\R^d} \mathbbm{1}_{A^c}(x,\gamma) \mu(\d \gamma) \dd x\\
		& = \int_{\cC_T} \mathscr{L}^d\big( \big\{ x\in\R^d: \text{Card}(S(x;0,b,\gamma))\neq1 \big\} \big) \dd \gamma
	\end{align*}
	where all integrals are well-defined since the sections of $A^c$ w.r.t. $x$ (resp. $\gamma$) are measurable for $\mathscr{L}^d$-a.e. $x$ (resp. $\mu$-a.e. $\gamma$). In particular, from the above chains of identities we see that $\bar\mu(A^c)=0$ if and only if the last integral is $0$, which corresponds to statement $b)$, and if and only if
	\begin{align*}
		\int_{\R^d} & \mu\big( \big\{ \gamma\in \cP(\cC_T): \text{Card}(S(x;0,b,\gamma))\neq1 \big\} \big) \dd x = 0\\
		& \Leftrightarrow \mu\big(\big\{ \gamma\in \cP(\cC_T): \text{Card}(S(x;0,b,\gamma))\neq1 \big\} \big)= 0 \text{ for }\mathscr{L}^d\text{-a.e. } x
	\end{align*}
	which corresponds to statement $a)$.
\end{proof}

The next statement provides path-by-path uniqueness for Lebesgue a.e. $x\in \R^3$, for any $\mu\in\cP(\cC_T)$, whenever $b=u$ is a Leray solution to the $3$D Navier--Stokes equations; it contains Theorem \ref{thm_main1_intro} as a particular subcase with $\mu=\delta_0$ (i.e. $\gamma\equiv 0$).

\begin{theorem}\label{thm:main1}
	Let $u_0\in L^2_\sigma$, $u$ be an associated Leray solution and $\mu\in\cP(\cC_T)$. Then path-by-path a.e. uniqueness holds for $(u,\mu)$; moreover, for Lebesgue a.e. $x\in\R^3$, path-by-path uniqueness holds for $(x,u,\mu)$.
\end{theorem}

\begin{proof}
By combining Corollaries \ref{cor:onesided_navier_stokes} and \ref{cor:asymmetric_lusin_lipschitz_flow}, we deduce trajectorial a.e. uniqueness for any fixed $\gamma\in \cC_T$:
	\begin{align*}
	\text{Card}(S(x;0,u,\gamma))=1 \text{ for Lebesgue a.e. }x\in\R^3.
	\end{align*}
As the statement holds for \emph{any} $\gamma$, it holds in particular on $\text{supp}\,\mu$, thus verifying the definition of path-by-path a.e. uniqueness holds for $(u,\mu)$.
In order to conclude, by Lemma \ref{lem:relations_uniqueness}, it suffices to verify that the set $A$ as defined therein is $\cA$-measurable; since $\cA$ is complete, it suffices to show the existence of another measurable set $B$ such that $B\subset A$ and $\bar\mu(B)=1$.

Consider the random solution map $(x,\gamma)\mapsto X_t(x,\gamma)$ (cf. Remark \ref{rem:random.representation}), which again is well-defined (up to $\bar\mu$-negligible sets) thanks to the regularity of $u$. Let $h\in L^1_T L^{3,\infty}$ be the function from Corollary \ref{cor:onesided_navier_stokes}, then it follows from the proof of Corollary \ref{cor:asymmetric_lusin_lipschitz_flow} that $A\subset B$ for
\begin{align*}
	B:=\Big\{(x,\gamma): X(x,\gamma)\in S(x;0,u,\gamma), \, \int_0^T h(X_s(x,\gamma)) \dd s<\infty \Big\}.
\end{align*}
Indeed, for any such $(x,\gamma)$, there exists at least one solution given by $X(x,\gamma)$; if $z$ is any other solution, by the asymmetric Lusin-Lipschitz property, we have
	\begin{align*}
		|z_t-X_t(x,\gamma)| \leq \int_0^t |u_s(z_s)-u_s(X_s(x,\gamma))| \dd s \leq \int_0^t |h_s(X_s(x,\gamma))|z_s-X_s(x,\gamma)| \dd s
	\end{align*}
	for all $t\geq 0$; so for $(x,\gamma)\in B$, by Gr\"onwall we can conclude that $z$ must coincide with $X(x,\gamma)$, namely there is exactly one solution.

Letting $u^n$ be smooth-in-space approximations of $u$ obtained by mollification; denoting by $X^n(x,\gamma)$ the associated solutions (which are pointwise defined), using the Sobolev regularity of $u$ and integrating the bound
\eqref{eq:stability_RLF} (for $p=2$) w.r.t. $\mu$ we find
\begin{align*}
	\int_{\R^3\times \cC_T} (1\wedge \| X^n(x,\gamma)-X(x,\gamma)\|_{\cC_T}^2 (1+|x|)^{-4} \bar\mu(\dd x,\dd \gamma)
		\lesssim e^{2\lambda} \| u^n-u\|_{L^1_T L^2}^2 + \frac{1}{\lambda^2} \| \nabla u\|_{L^1_T L^2}^2;
\end{align*}
by the arbitrariness of $\lambda$ we have
\begin{align*}
	\lim_{n\to\infty} \int_{\R^3\times \cC_T} (1\wedge \| X^n(x,\gamma)-X(x,\gamma)\|_{\cC_T}^2 (1+|x|)^{-4} \bar\mu(\dd x,\dd \gamma) =0,
\end{align*}
namely convergence in measure holds. We can therefore extract a (not relabelled) subsequence such that $X^n(x,\gamma)\to X(x,\gamma)$ in $\cC_T$ for $\bar{\mu}$-a.e. $(x,\gamma)$; since $X^n(x,\gamma)\in S(x;0,u^n,\gamma)$ and $u^n\to u$ in $L^1_T C^0_\loc$, we conclude that for any such $(x,\gamma)$ we have $X(x,\gamma)\in S(x;0,u,\gamma)$.
Similarly, defining $H_T(x,\gamma):=\int_0^T h(X_s(x,\gamma)) \dd s$, performing estimates as in Corollary  \ref{cor:asymmetric_lusin_lipschitz_flow} at fixed $\gamma$ and integrating w.r.t. $\mu$, we find that
\begin{align*}
	\bar{\mu}\big(\big\{ (x,\gamma):\, H_T(x,\gamma)> C\big\}\big)
	& = \int_{\cC_T} \mathscr{L}^d\big( \big\{x:\, H_T(x,\gamma)> C\big\}\big) \mu(\dd\gamma)\\
	& \leq C^{-3} \int_{\cC_T} \| H_T(\cdot,\gamma)\|_{L^3,\infty}^3 \mu(\dd\gamma)
	\lesssim C^{-3} \Big( \int_0^T \| \nabla u_s\|_{L^{3,1}} \dd s\Big)^3;
\end{align*}
by the arbitrariness of $C>0$, we deduce that $\bar\mu(B)=1$, concluding the proof.
\end{proof}

Theorem \ref{thm:main1} covers any possible driving process $\gamma$.
For convenience, we now specialize its statement to the case of $\gamma=\eps W$ with Brownian $W$; the first part of Theorem \ref{thm_main2_intro} is covered by Theorem \ref{thm:main.intermediate} below. 

\begin{theorem}\label{thm:main.intermediate}
	Let $u_0\in L^2_\sigma$, $u$ be an associated Leray solution and let $\eps>0$. Then for Lebesgue a.e. $x\in \R^3$, strong existence, pathwise uniqueness and path-by-path uniqueness hold for the SDE
	\begin{equation}\label{eq:SDE}
		Y_t =x + \int_0^t u_s(Y_s) \dd s + \eps \dd W_t \quad \forall\, t\geq 0.
	\end{equation}
\end{theorem}

\begin{proof}
	Taking $\mu=Law(\eps W)$ in Theorem \ref{thm:main1}, we immediately deduce that for Lebesgue a.e. $x\in\R^3$, path-by-path uniqueness holds for the SDE problem \eqref{eq:SDE}. This implies pathwise uniqueness as well.
	On the other hand, since $u\in L^1_T C^0$, for \emph{every} $x\in\R^3$, it's easy to construct weak solutions to \eqref{eq:SDE} by compactness arguments. As a consequence, on the set of full Lebesgue measure of $x\in\R^3$ on which pathwise uniqueness and weak existence hold, we deduce strong existence by the Yamada--Watanabe theorem.	
\end{proof}

Finally, under additional regularity on $u_0$, we can exploit the non-degeneracy of the noise to strengthen the result, establishing path-by-path uniqueness for \emph{every} initial $x$. Theorem \ref{thm_main2_intro} is then a consequence of a combination of Theorem \ref{thm:main.intermediate} and Theorem \ref{thm:main2} below.

\begin{theorem}\label{thm:main2}
	Let $u_0\in L^2_\sigma\cap H^{1/2}$, $u$ be an associated Leray solution and $\eps>0$; let $\mu=Law(\eps W)$, where $W$ is Brownian motion. Then for every $x\in\R^3$, strong existence and path-by-path uniqueness holds for $(x,u,\mu)$.
\end{theorem}

\begin{proof}
	As before, since $u\in L^1_T C^0$, weak existence holds by compactness; once we have shown path-by-path uniqueness, strong existence will follow from the Yamada--Watanabe theorem.
	 
	For notational simplicity we only consider $\eps=1$, the general case being similar. Let $\mu=Law(W)$ denote the Wiener measure; we can then assume $W_t(\gamma)=\gamma_t$ to be the canonical process on the Wiener space $(\cC_T,\cF_t,\mu)$, where $\gamma\in \cC_T$ and $\cF_t$ is the standard argumentation of the natural filtration generated by $W$. We divide the proof in several steps.
	
	\emph{Step 1: uniqueness at small times.}
	Since $u_0\in H^{1/2}$, we can invoke \cite[Thm. 1.1]{CheLer1995} to deduce the existence of $t^\ast>0$ sufficiently small such that $u$ is logLipschitz on $[0,t^\ast]$, in the sense that
	\begin{align*}
		\int_0^{t^\ast} \sup_{x\neq y} \frac{|u_t(x)-u_t(y)|}{w(|x-y|)} \dd t <\infty
	\end{align*}
	for a modulus of continuity $w(r)$ that behaves like $r|\log r|$ for $r\ll 1$.
	By Osgood's lemma, it follows that the SDE is well-posed on $[0,t^\ast]$; this is a purely analytical result valid for any realisation $\gamma\in\cC_T$, thus a path-by-path uniqueness result.
	
	\emph{Step 2: uniqueness on $[t^\ast,T]$.}
	Let us explain the strategy somewhat loosely. Suppose we are given two distinct solutions $z^1$, $z^2$ starting from $x$; by Step 1, they coincide on $[0,t^\ast]$, so we can restrict to $[t^\ast,T]$.
	By Corollary \ref{cor:onesided_navier_stokes}, $u$ satisfies the asymmetric Lusin-Lipschitz property for some $h\in L^1_T L^{3,\infty}$, therefore
	\begin{align*}
		|z^1_t-z^2_t| \leq \int_{t^\ast}^t |u_s(z^1_s)-u_s(z^2_s)| \dd s \leq \int_{t^\ast}^t |h_s(z^1_s)|z^1_s-z^2_s| \dd s\quad \forall\, t>t^\ast;
	\end{align*}
	by Gr\"onwall, we can conclude that $z^1\equiv z^2$ on $[t^\ast,T]$, if we can guarantee that
	\begin{equation}\label{eq:requirement}
		\int_{t^\ast}^T h_s(z^1_s) \dd s <\infty.
	\end{equation}
	This is now a property that depends only on $z^1$, but not $z^2$; in other words, for fixed $\gamma\in\cC_T$ , if there exists a solution $z\in S(x;0,\gamma,u)$ such that \eqref{eq:requirement} holds, then necessarily any other solution must coincide with it. In particular, proving path-by-path uniqueness for $(x,u,\mu)$ reduces to showing that
	\begin{equation}\label{eq:requirement_2}
		\mu\Big(\gamma\in \cC_T: \exists\, z\in S(x;0,u,\gamma) \text{  s.t. } \int_{t^\ast}^T h_s(z_s) \dd s<\infty\Big) = 1.
	\end{equation}
	A natural guess is to take $z_t=X_t(x,\gamma)=\Phi_t(x;u,\gamma)$, where $\Phi$ is the solution map constructed in Theorem \ref{thm:existence_RLF_gamma}; but a priori we only know that $\int_0^T h_s(\Phi_t(y;u,\gamma)<\infty$ for Lebesgue a.e. $y$, not necessarily the specific $x$ in consideration. To prove \eqref{eq:requirement_2}, we need to use the fact the diffusive nature of the stochastic dynamics on $[0,t^\ast]$, coupled with the incompressibility of the flow $\Phi$ restarted after $t^\ast$.
	
	\emph{Step 3: rigorous verification of \eqref{eq:requirement_2}.}
	On $[0,t^\ast]$, $\rho_t:=\Law(X_t)$ solves the Fokker-Planck equation
	\begin{align*}
		\partial_t \rho + \nabla\cdot( u\, \rho) = \frac{1}{2}\Delta \rho.
	\end{align*}
	Since $u\in L^2_T H^1$, $\nabla\cdot u=0$, we can invoke the results from \cite{ZhaZha2021} to deduce that $\rho \in L^1([0,t^\ast]; W^{1,p})$ for any $p<3/2$. In particular, up to possibly redefining $\tilde t^\ast\leq t^\ast$, we can assume that $\rho_{t^\ast}=Law(X_{t^\ast})\in W^{1,p}$, so that in particular $\rho_{t^\ast}\ll \mathscr{L}^3$.
	
	Since $\mu$ is the Wiener measure, $(\theta_{t^\ast} \gamma)= \gamma_{\cdot+t^\ast}-\gamma_{t^\ast}$ is again distributed as $\mu$ and is independent of $\cF_{t^\ast}$, $\theta_{t^\ast} \gamma\perp \cF_{t^\ast}$. We now define a process $Z$ on $[0,T]$ by
	\begin{equation*}
	Z_t(\gamma) :=\begin{cases}
		X_t(\gamma)\qquad & \text{if } t\in [0,t^\ast]\\
		\Phi_{t-t^\ast}(X_{t^\ast}(\gamma); \tau_{t^\ast} u,\theta_{t^\ast} \gamma)& \text{if } t\in [t^\ast,T]
	\end{cases}.
	\end{equation*}
We claim that this is a well-defined process, and $\mu$-a.s. a solution to the random ODE, thanks to the properties of $\Phi$ and the independence of $X_{t^\ast}$ from $\theta_{t^\ast} \gamma$. Indeed
\begin{align*}
	\mu\Big(Z(\gamma)\in S(x;0,u,\gamma)\Big)
	& = \mu\Big(Z(\gamma)\in S(X_{t^\ast};t^\ast,u,\gamma_\cdot-\gamma_{t^\ast})\Big) \\
	& = \mu\Big( \Phi(X_{t^\ast}(\gamma); \tau_{t^\ast}u, \theta_{t^\ast} \gamma)\in S(X_{t^\ast}; 0,\tau_{t^\ast} u,\theta_{t^\ast}\gamma)\Big)\\
	& = \int_{\cC_T} \int_{\R^3} \mathbbm{1}_{\Phi(y; \tau_{t^\ast}u, \tilde\gamma)\in S(y; 0,\tau_{t^\ast} u,\tilde\gamma)} \rho_{t^\ast}(y) \dd y \,\mu(\dd \tilde\gamma)\\
	& = \int_{\cC_T} 1 \,\mu(\dd \tilde\gamma) = 1,
\end{align*}
where we used the fact that, for fixed $\tilde \gamma$, $\Phi(y; \tau_{t^\ast}u, \tilde\gamma)\in S(y; 0,\tau_{t^\ast} u,\tilde\gamma)$ for Lebesgue a.e. $y$ and $\rho_{t^\ast}\ll\mathscr{L}^3$.
	The verification of \eqref{eq:requirement_2} is similar:
\begin{align*}
	\mu\Big( \int_{t^\ast}^T h_r(Z_r(\gamma)) \dd r <\infty\Big)
	& = \mu\Big( \int_0^{T-t^\ast} h_{t^\ast+r}(\Phi_r(X_{t^\ast}(\gamma); \tau_{t^\ast}u, \theta_{t^\ast} \gamma)) \dd r <\infty\Big)\\
	& = \int_{\cC_T} \int_{\R^3} \mathbbm{1}_{\int_0^{T-t^\ast} h_{t^\ast+r}(\Phi_r(y; \tau_{t^\ast}u, \tilde \gamma)) \dd r <\infty}\, \rho_{t^\ast}(y) \dd y \,\mu(\dd \tilde\gamma)\\
	& = \int_{\cC_T} 1 \,\mu(\dd \tilde\gamma) = 1,
\end{align*}
	where we used the fact that, for fixed $\tilde\gamma$, $\int_0^T h_{t^\ast+r}(\Phi_r(y; \tau_{t^\ast}u, \tilde\gamma) \dd r <\infty$ for Lebesgue-a.e. $y$, by arguing as in Corollary \ref{cor:asymmetric_lusin_lipschitz_flow}, and $\rho_{t^\ast}\ll\mathscr{L}^3$.
\end{proof}

\begin{remark}\label{rem:initial data}
	Theorem \ref{thm:main2} requires $u_0\in H^{1/2}$ to rely on the short-time logLipschitz regularity result from \cite{CheLer1995}.
	We expect other classes of critical initial data to work as well, like $u_0\in B^{-1+d/p}_{p,\infty}$, or even $u_0\in VMO^{-1}$ in light of \cite{KocTat2001}, up to establishing similar regularity results as \cite{CheLer1995}.
\end{remark}

A natural question left open by Theorem \ref{thm:main2} is the following.

\begin{conjecture}\label{conjecture}
	For any $u_0\in L^2_\sigma$ and any associated Leray solution $u$, pathwise uniqueness holds for the SDE for every $x\in\R^3$.
\end{conjecture}

\begin{remark}
	Up to technical modifications, under the assumptions of Theorem \ref{thm:main2}, one can show convergence of Picard iterations similarly to Theorem \ref{thm:picard} (here one needs to split the analysis between $[0,t^\ast]$ and $[t^\ast,T]$).
\end{remark}

\subsection{Consequences}\label{subsec:consequences}

Given a Leray solution $u$, we denote by $\cR\subset [0,+\infty)$ its set of regular times: $t_0\in \cR$ if and only if $t\mapsto\| \nabla u(t)\|_{L^2}$ is essentially bounded (in which case, $u$ is actually smooth) around $t_0$. $\mathcal{I}:=[0,+\infty)\setminus \cR$ is the set of singular times of $u$; by definition $\cR$ is open, while it is known that $\mathcal{I}$ is compact and has Hausdorff dimension at most $1/2$ (in particular, it has zero Lebesgue measure), cf. \cite[Ch. 8]{RRS2016} .


\begin{lemma}\label{lem:moment-convergence}
	Let $u_0\in L^2_\sigma$, $u$ an associated Leray solution and $\cR$ its regular set; let $W$ be Brownian motion.
	Then for any $s\in \cR$ and any $x\in \R^3$, there exists a unique strong solution to the SDE (in integral form)
	\begin{equation}\label{eq:SDE_consequences}
		X_t = x + \int_s^t u_r(X_r) \dd r + W_t-W_s \quad \forall\, t\geq s. 
	\end{equation}
	Denote such solution by $X_{s\to t}(x)$. Then for any $T<\infty$, the map $x\mapsto X_{s\to \cdot}(x)$ is stochastically continuous from $\R^3$ to $\cC_T$ and for any $m\in [1,\infty)$, it holds
	\begin{equation}\label{eq:moment-convergence}
		\lim_{y\to x} \E\Big[ \| X_{s\to s+\cdot}(y) - X_{s\to s+\cdot}(x)\|_{\cC_T}^m \Big]= 0.
	\end{equation}
\end{lemma}

\begin{proof}
	Since $s\in\cR$, by shifting $u_{s+\cdot}$ is again a Leray solution, started at $u_s\in H^1$. Wellposedness of the SDE thus follows from Theorem \ref{thm:main2}.
	
	We may assume $s=0$ and $u_0\in H^1$, as the general case follows by shifting.
	Since $u\in L^1_T C^0$ and $W$ is Brownian, we have uniform-in-$y$ bounds on $\E[ \|X(y)\|_{\cC_T}^m]$; by Vitali's theorem, \eqref{eq:moment-convergence} will follow once we show convergence in probability as $y\to x$.
		
	The last claim can be established by standard compactness arguments and the Gyongy--Krylov lemma \cite{GyoKry1996}, relying on the fact that uniqueness for the SDE is known: since $u\in L^1_T C^0$ and $W$ is H\"older, by Ascoli-Arzelà the family $\{(X(y),X(x))\}_{y\in\R^3}$ is tight in $\cC_T$, and all its accumulations points as $y\to x$ will be distributed processes of the form $(Z^1,Z^2)$, where $Z^i$ are both solutions to the SDE started at $x$. But then necessarily $Z^1=Z^2$, implying that $\| X(y)-X(x)\|_{\cC_T}\to 0$ in probability as $y\to x$.
\end{proof}

The next result loosely speaking states that the two-parameter semigroup $P_{s\to t}$ associated to the SDE \eqref{eq:SDE_consequences} is well-defined, whenever $s\in\cR$, and satisfies both the strong Feller property and an a.s. Markov property. This type of result is not completely new in the literature, c.f. \cite[Theorem 1.1-(ii)]{ZhaZha2021}; quite nicely, in our statement the Lebesgue negligible set where such properties fail is only in the time index, and can be directly related to the singular set $\mathcal{I}$ of the Leray solution.

\begin{corollary}
	Let $u_0\in L^2_\sigma$, $u$ be a Leray solution. For any $s\in\cR$ and any $t>s$, define the map
	\begin{align*}
		P_{s\to t} \varphi(x) := \E[\varphi(X_{s\to t} (x))] \quad \forall\, \varphi\in \mathcal{B}_b
	\end{align*}
	where $\cB_b=\cB_b(\R^3;\R)$ denotes the set of bounded, Borel measurable functions on $\R^3$. Then:
	\begin{enumerate}
	\item For any $s\in\cR$, $P_{s\to t}$ maps $\cB_b$ into $C^0$ (\emph{almost sure strong Feller property}).
	\item For any $s,r\in \cR$ and $t\in (0,+\infty)$ such that $s<r<t$, it holds $P_{s\to t} = P_{s\to r} \circ P_{r\to t}$ (\emph{almost sure Markov property}).
	\end{enumerate}
\end{corollary}

\begin{proof}
	First notice that, by Lemma \ref{lem:moment-convergence}, for any $s\in\cR$ and any $t>s$, $P_{s\to t}$ is a linear map from $C^0$ to itself; additionally, since pathwise uniqueness for \eqref{eq:SDE_consequences} holds for every $x\in \R^3$, standard disintegration arguments imply the same statement for random initial conditions $X$ which are independent of the driving noise $W_\cdot-W_s$.
	
	Let us first prove (2). Since the process $X_{s\to \cdot}(x)$ solves the SDE on $[s,t]$, it is also a solution on $[u,t]$, with initial condition at time $u$ given by $X_{s\to r}(x)$ and driving noise $W_\cdot -W_r$. Since it is a strong solution and $W$ is Brownian, $X_{s\to r}(x)$ is $\cF_r$-adapted and $W_\cdot -W_r$ is independent of $\cF_r$; it follows that $X_{s\to t}(x) = X_{r\to t}(X_{s\to r}(x))$ and upon taking conditional expectation
	\begin{align*}
		\E[\varphi(X_{s\to t}(x))|\cF_r] = (P_{r\to t}\varphi) (X_{s\to r}(x)).
	\end{align*}
	Taking expectation again yields (2).
	
	To prove (1), notice that given $s\in\cR$ we can find $r\in (s,t)$ small enough such that $[s,r]\subset \cR$; in particular, $u$ is smooth on $[s,r]$. In view of the relation $P_{s\to t}=P_{s\to r}\circ P_{r\to t}$ and the fact that $P_{r\to t}$ is bounded from $\cB_b$ to itself, we only need to show that $P_{s\to r}$ maps $\cB_b$ into $C^0$. Since on $[s,r]$, $X$ is the solution to an SDE with smooth drift $u$, the latter property is classical, see for instance \cite[Ch. 12]{daprato2014}.
\end{proof}

\begin{remark}[Time reversal]
	Similar results as the ones considered above apply to the backward trajectories $X_{s\leftarrow t}(x)$: indeed, by time reversal, for fixed $t>0$, $s\mapsto X_{s\leftarrow t}(x)=\tilde X_s(x)$, where $\tilde X$ solves a forward SDE on $[0,t]$ associated to $\tilde u_r(x)=-u_{t-r}(x)$ and reversed BM $\tilde W_r= W_t-W_{t-r}$. In particular, $\tilde u$ inherits the asymmetric Lusin--Lipschitz property and the regular set $\cR$ from $u$, so that the same arguments apply. Solvability of backward trajectories is relevant for the Feynman--Kac representation $f_t(x)=\E[f_0(X_{0\leftarrow t})(x)]$ of solutions $f$ to the advection-diffusion PDE
	\begin{align*}
		\partial_t f + u\cdot\nabla f= \frac{1}{2}\Delta f, \quad f\vert_{t=0}=f_0.
	\end{align*}
\end{remark}

\begin{remark}\label{rem:regularity_gradient}
	Even though Theorem \ref{thm:main2} provides well-posedness for the SDE, it does not provide any useful estimate for the gradient of the flow $\nabla X_{s\to t}(x)$. If these were available, one could hope to use Constantin-Iyer's formula \eqref{eq:constantin_iyer}-\eqref{eq:constantin_iyer2} to bootstrap regularity for $u$, cf. \cite{RocZha2021}. Since $\nabla X$ solves
	\begin{align*}
		\frac{\dd}{\dd t} \nabla X_{s\to t}(x)=\nabla u_t(X_{s\to t}(x)) \nabla X_{s\to t}(x)
	\end{align*}
	where $\nabla u\in L^1_T L^3$, using the incompressibility of the Random RLF one can deduce a ``logarithmic'' regularity estimate typical of such flows (cf. \cite{Leger2018,BruNgu2021}). For instance, employing a  Gr\"onwall estimate one can find
	\begin{align*}
		\Big\| \sup_{t\in [0,T]}\log |\nabla X_{0\to t}| \Big\|_{L^3} \leq \| \nabla u\|_{L^1_T L^3}
	\end{align*}
	where the estimate is uniform among all Brownian realizations $\gamma$; alternatively one can derive the asymmetric Lusin--Lipschitz property \eqref{eq:asymmetric_lusin_flow} from Corollary \ref{cor:asymmetric_lusin_lipschitz_flow}. Both estimate are rather poor and do not exploit the presence of Brownian diffusion, but merely the Sobolev regularity of $u$; we leave the question of obtaining sharper regularity results for future research.
\end{remark}



\section*{Acknowledgements}
I'm very thankful to Maria Colombo and Elia Bruè for stimulating discussions and advices on this work.

\section*{Funding information}
The author is supported by the SNSF Grant 182565 and by the Swiss State Secretariat for Education, Research and Innovation (SERI) under contract number MB22.00034 and by the Istituto Nazionale di Alta Matematica (INdAM) through the project GNAMPA 2025 “Modelli stocastici in Fluidodinamica e Turbolenza”.

\bibliography{biblio}{}
\bibliographystyle{plain}

\end{document}